\documentclass{article}
\usepackage[utf8]{inputenc}
\usepackage[hidelinks]{hyperref}
\usepackage[misc]{ifsym}
\newcounter{newcounter}
\usepackage{centernot}
\usepackage{cancel}

\newcommand{\institute}[1]{\\{\scriptsize
		\begin{tabular}[t]{@{\footnotesize}c@{}}#1\end{tabular}}}
\newcommand{\email}[1]{\\{\scriptsize\tt #1}}
\title{Classifying different criteria for learning algebraic structures}
\author{
	\small{Nikolay Bazhenov$^{2,5}$}
	\email{nickbazh@yandex.ru}
	\and
	\small{Vittorio Cipriani$^{6}$}
	\email{vittorio.cipriani17@gmail.com}
	\and
	\small{Sanjay Jain$^3$}
	\email{sanjay@comp.nus.edu.sg}
	\and \ \ \ \ \ \  \ 
	\small{Luca San Mauro$^1$}
\ \ \  \ \ \ 	\email{luca.sanmauro@gmail.com}
\and \and \ 
	\small{Frank Stephan$^{3,4}$}
	\email{fstephan@comp.nus.edu.sg}
	\and
	\vspace{-1cm}
	\and
	\institute{
		$^1$ Department of Philosophy, University of Bari\\
	    $^2$ Kazan Federal University, Kazan, Russia\\
	    $^3$ School of Computing, National University of Singapore\\
	    $^4$ Department of Mathematics, National University of Singapore\\
		$^5$ Sobolev Institute of Mathematics, Novosibirsk, Russia\\
		$^6$ Institute of Discrete Mathematics and Geometry, Technische Universit{\"a}t Wien, Austria
	}
}
\date{}
\usepackage[margin=3.2cm]{geometry}

\usepackage{amssymb}
\usepackage{amsmath}
\usepackage{amsthm}
\usepackage{float}
\usepackage[shortlabels]{enumitem}
\usepackage{graphicx}
\usepackage{theoremref}
\usepackage{tikz}

\usetikzlibrary{backgrounds}
\usetikzlibrary{arrows}
\usetikzlibrary{arrows.meta}
\usetikzlibrary{shapes,shapes.geometric,shapes.misc}
\usetikzlibrary{fit}
\usetikzlibrary{positioning}
\usetikzlibrary{calc}
\usetikzlibrary{quotes}

\tikzstyle{tikzfig}=[baseline=-0.25em,scale=0.5]
\pgfdeclarelayer{edgelayer}
\pgfdeclarelayer{nodelayer}
\pgfsetlayers{background,nodelayer,edgelayer,main}
\tikzstyle{none}=[inner sep=0mm]
\tikzstyle{every loop}=[]
\tikzset{>={latex[width=1mm,length=1mm]}}
\newcommand{\drawThm}[3]{ \draw #1 
	\ifx&#2&
	\else
	node[style=thmref, pos=0.01] {\resizebox{5mm}{3mm}{#2}} 
	\fi
	\ifx&#3&
	\else
	node[style=thmref, pos=0.99] {\resizebox{5mm}{3mm}{#3}}
	\fi ;
}

\tikzstyle{box}=[fill={rgb,255: red,228; green,228; blue,228}, draw=black, shape=rectangle]

\tikzstyle{reducible}=[->, dashed]
\tikzstyle{strictreducible}=[->]
\tikzstyle{nonreducible}=[->, draw=red]
\tikzstyle{uncomp}=[draw=red, <->]

\tikzstyle{thmref}=[opacity=0,inner sep=2mm]
\usepackage{ifthen}
\usepackage{url}
\usepackage{doi}
\makeatletter
\newcommand{\Fin}{\mathbf{Fin}}
\newcommand{\K}{\mathfrak{K}}

\newcommand{\defas}{:=}
\DeclareMathOperator{\A}{\mathcal{A}}
\DeclareMathOperator{\esse}{\mathcal{S}}
\DeclareMathOperator{\B}{\mathcal{B}}
\newcommand{\baire}{\nats^{<\nats}}
\newcommand{\nats}{\mathbb{N}}
\newcommand{\ldofk}{\mathrm{LD}(\K)}

\newcommand{\col}{\mathbf{co}}

\newcommand{\learnerM}{\mathbf{M}}
\newcommand{\Lomega}{\mathcal{L}_{\omega_1\omega}}
\newcommand{\equalitynats}{=_\nats}
\DeclareMathOperator{\range}{range}
\newcommand{\ex}{\mathbf{Ex}}
\newcommand{\code}[1]{\ulcorner #1 \urcorner}
\newcommand{\learnreducible}[1]{\leq_{\mathrm{Learn}}^{#1}}
\newcommand{\notlearnreducible}[1]{\not\leq_{\mathrm{Learn}}^{#1}}
\newcommand{\strictlearnreducible}[1]{<_{\mathrm{Learn}}^{#1}}
\newcommand{\strictlearnincomparable}[1]{|_{\mathrm{Learn}}^{#1}}
\newcommand{\function}[3]{#1 : #2 \rightarrow #3}
\newcommand{\Baire}{\nats^\nats}
\newcommand{\learnequiv}[1]{\equiv_{\mathrm{Learn}}^{#1}}
\newcommand{\Id}{Id}
\newcommand{\Cantor}{2^{\nats}}
\newcommand{\cantor}{2^{<\nats}}
\newcommand{\nonushape}{\mathbf{nUs}}

\DeclareMathOperator*{\bigdoublewedge}{\bigwedge\mkern-15mu\bigwedge}
\DeclareMathOperator*{\bigdoublevee}{\bigvee\mkern-15mu\bigvee}

\newcommand{\sigmainf}[1]{\Sigma_{#1}^{\mathrm{inf}}}

\newcommand{\piinf}[1]{\Pi_{#1}^{\mathrm{inf}}}

\newcommand{\str}[1]{\langle #1 \rangle}

\newcommand{\bc}{\mathbf{Bc}}
\newcommand{\pl}{\mathbf{PL}}

\newcommand{\decisive}{\mathbf{Dec}}
\newcommand{\finitary}{\mathrm{fin}}
\newcommand{\thsigma}[2]{\mathrm{Th}_{\Sigma_{#1}^\mathrm{inf}}(#2)}

\newcommand{\erange}{E_{range}}
\newcommand{\eset}{E_{set}}

\newtheorem{theorem}{Theorem}[section]
\newtheorem{proposition}[theorem]{Proposition}
\newtheorem{lemma}[theorem]{Lemma}
\newtheorem{corollary}[theorem]{Corollary}
\theoremstyle{definition}
\newtheorem{definition}[theorem]{Definition}

\newtheorem{remark}[theorem]{Remark}

\begin{document}
	\maketitle
	{\let\thefootnote\relax\footnote{{

  The work of Bazhenov was supported by the Russian Science Foundation (project no. 24-11-00227). Cipriani was supported by the Austrian Science Fund FWF, Project P 36781. S.~Jain and F.~Stephan were supported by Singapore Ministry of Education (MOE) AcRF Tier 2 grant MOE-000538-00. Additionally, S.~Jain was supported by NUS grant E-252-00-0021-01. San Mauro is a member of INDAM-GNSAGA.
  }}}
		
		\begin{abstract}
			In the last years there has been a growing interest in the study of learning problems associated with algebraic structures. The framework we use models the scenario in which a learner is given larger and larger fragments of a structure from a given target family and is required to output an hypothesis about the structure's isomorphism type. So far researchers focused on $\ex$-learning, in which the learner is asked to eventually stabilize to the correct hypothesis, and on restrictions where the learner is allowed to change the hypothesis a fixed number of times. Yet, other learning paradigms coming from classical algorithmic learning theory remained unexplored. We study the \lq\lq learning power\rq\rq\ of such criteria, comparing them via descriptive-set-theoretic tools thanks to the novel notion of $E$-learnability.  The main outcome of this paper is that such criteria admit natural syntactic characterizations in terms of infinitary formulas analogous to the one given for $\ex$-learning in \cite{bazhenov2020learning}. Such characterizations give a powerful method to understand whether a family of structure is learnable with respect to the desired criterion.
		\end{abstract}
		%\tableofcontents

		\section{Introduction}
		This paper aims to advance the study of algorithmic learning theory for algebraic structures considering new learning criteria and providing a syntactic characterization of them.
		
		Algorithmic learning theory has its roots in the work of Gold \cite{Gold67} and Putnam \cite{putnam1965trial} in the 1960s and it encompasses various formal frameworks for the inductive inference. In a broad sense, this research program models how a learner might acquire systematic knowledge about a given environment by accessing growing volumes of data. Classical paradigms primarily concentrated on inferring formal languages or computable functions (see, e.g., \cite{lange2008learning,zz-tcs-08}). In order to understand which families can be learnt with respect to a given paradigm, researchers focused on combinatorial characterizations of learnability (see e.g.,(\cite{ANGLUIN1980117}).
		
		Recently, there has been an increase of interest in learning data that carries structural content, with a focus on well-known classes of algebraic structures, like vector spaces, rings, trees, and matroids, (see, e.g., \cite{SV01,MS04,HaSt07,GaoStephan-12,FKS-19}). The framework we use is defined in Section \ref{sec:our_learning_paradigm} and was introduced and later refined in a series of papers \cite{bazhenov2020learning,bazhenov2021learning}. It draws upon concepts and techniques from computable structure theory: in a nutshell, a \emph{learning problem} consists of a countable family $\K$ of nonisomorphic countable structures; a \emph{learner}\ is an agent provided with increasingly larger portions of an isomorphic copy of a structure $\esse$ from $\K$ and, at each stage, is required to output a conjecture about the isomorphism type of $\esse$. We highlight that the learner has no complexity- or computability-theoretic restrictions (we refer the reader interested in how such restrictions affect the learning framework to \cite{bazhenov2021turing}). 
		
		So far researchers in this area mostly focused on $\ex$-learnability: A family of structures is \emph{$\ex$-learnable} if there exists a learner that, in the limit, stabilizes to the correct conjecture. \cite[Theorem 3.1]{bazhenov2020learning} gives a nice syntactic characterization of $\ex$-learnability in terms of $\sigmainf{2}$ \emph{infinitary formulas} (introduced in Section \ref{sec:infinitary}). That is, a family of structure $\K:=\{\A_i:i \in \nats\}$ is $\ex$-learnable if and only if there exist $\sigmainf{2}$ formulas $\{\varphi_i : i \in \nats \}$ such that
		$\A_i\models \varphi_j \Leftrightarrow i=j$.

		We aim to consider other learning criteria, and we show that they admit \lq\lq nice\rq\rq\ syntactic characterizations. The fact that we could find such natural characterization support the claim that the learning paradigms we are considering (many of which come from classical algorithmic learning theory) are indeed natural.
		
		\thref{theorem:sigma1strongantichain} shows that, considering $\sigmainf{1}$ infinitary formulas instead of $\sigmainf{2}$ infinitary formulas, one obtains a syntactic characterization for learning without
		mind changes, also known as $\Fin$-learnability, analogous
		to the one given for $\ex$-learnability.

		The characterizations of $\ex$-learnability and $\mathbf{Fin}$-learnability described so far are obtained considering partial orders whose elements are the structures in the family ordered with respect to the inclusion of the $\sigmainf{1}$- and $\sigmainf{2}$-theories, where, in general, the $\sigmainf{n}$-theory of a structure is the set of $\sigmainf{n}$ formulas that are true in the structure. For the case of $\ex$-learnability and $\Fin$-learnability we call the corresponding partial orders $\sigmainf{2}$- and $\sigmainf{1}$-\emph{strong antichains}.  
		
		Natural weakenings of $\sigmainf{1}$- and $\sigmainf{2}$-\emph{strong antichains} are $\sigmainf{1}$- and $\sigmainf{2}$\emph{-antichains}: here we ask that any two structures in the family are pairwise incomparable respectively to the inclusion of the $\sigmainf{1}$- and $\sigmainf{2}$- theories.
		
		Surprisingly, \thref{theorem:12antichains} and \thref{theorem:plcharacterization} show that these two partial orders exactly characterize two natural learning paradigm already considered in the context of classical inductive inference, namely \emph{co-learnability} (\thref{definition:colearning}) and \emph{partial learnability} (\thref{definition:pl}), denoted respectively by $\col$ and $\pl$. In the first, the only conjecture that a learner does not output is the correct one, whereas in the second, the only conjecture that a learner outputs infinitely many times is the correct one.
		
		As we are interested in understanding the behavior of learning paradigms coming from classical inductive inference, we also study two natural restrictions of $\ex$-learning namely $\nonushape$- and $\decisive$-learnability (\thref{definition:nonushapedec}), where $\nonushape$ stands for \emph{non U shaped} while $\decisive$ for \emph{decisive}. The converge criterion for the first one is the same of $\ex$-learning except that a learner cannot change its mind once it outputs for the first time the correct conjecture. Decisive learning restricts this behavior not allowing the learner to get back to a previously abandoned conjecture. It turns out that the two paradigms in our context learn the same families (\thref{proposition:nonushapedecequiv}) and \thref{theorem:solidposets} provides a characterization for such paradigms in terms of \emph{solid $\sigmainf{1}$-partial orders} (\thref{definition:solidposets}). Informally, a solid $\sigmainf{1}$\-partial order requires that any structure $\A$ in the family has a $\sigmainf{1}$-formula separating $\A$ from its lower cone, i.e., from the other structures in the family whose $\sigmainf{1}$ theory is properly contained in $\thsigma{1}{\A}$.

		The next natural question is to ask for what happens if we consider $\sigmainf{1}$- and $\sigmainf{2}$-\emph{partial orders}, in which we just ask that the inclusion of $\sigmainf{1}$- and $\sigmainf{2}$-theories is a partial order on the family. Notice that $\sigmainf{n}$-partial orders, in general, have been deeply investigated in \cite{ciprianimarconesanmauro}. It turns out that these paradigms do not come (up to our knowledge) from ones studied in the classical context but from the novel equivalence relations of \emph{$E$-learnability}: this notion has been introduced in \cite{bazhenov2021learning} to calibrate the complexity of nonlearnable families borrowing ideas from \emph{descriptive set theory} (we will say more about this in Section \ref{sec:elearnability}). Actually, we will be using a slightly different variation of $E$-learnability (\thref{definition:elearnability}): \thref{remark:elearning} motivates that our notion is more natural and corrects a mistake in \cite{bazcipsan_cie}. One of the aims of this paper is to investigate how this new notion of learnability relates to learning paradigms coming from inductive inference. The initial step in \cite{bazhenov2021learning} was to show that $\ex$-learnability has a natural descriptive set-theoretic interpretation, namely \cite[Theorem 3.1]{bazhenov2021learning} shows that a countable family of structures $\K$ is $\ex$-learnable if and only if the isomorphism relation
		associated with $\K$ is continuously reducible to the relation $E_0$ of eventual agreement on infinite sequence of natural numbers. Notice that $E_0$ has a pivotal role in the context of descriptive set theory (see e.g.\ \cite{HKL}): as already observed in \cite{bazhenov2021learning}, this may serve as evidence supporting the naturalness of the learning framework.
		This result unlocked a natural way of stratifying learning problems: it suffices to replace $E_0$ with equivalence relations of lower or higher complexity to obtain weaker or stronger notions of learnability. More precisely, we say that a family of structures $\K$ is \emph{$E$-learnable}, for an equivalence relation $E$, if there is a continuous reduction from the isomorphism relation associated with $\K$ to $E$ (see Section \ref{sec:elearnability}).
		
		Getting back to the question of which learning criteria are characterized by $\sigmainf{1}$- and $\sigmainf{2}$-partial orders, we obtain that these are respectively $\erange$-learnability and $\eset$-learnability (see Section \ref{sec:elearnability} for their definitions). The results about $\sigmainf{1}$- and $\sigmainf{2}$-partial orders were also obtained in \cite{ciprianimarconesanmauro}, but we provide an alternative proof of the correspondence between being $\erange$-learnable and being a $\sigmainf{1}$-partial orders in \thref{theorem:characterization_Erange}. Notice that $\erange$- and $E_0$-learnability are the first examples of two incomparable learning criteria in our learning hierarchy.
		
		The paper is organized as follows. Section \ref{sec:preliminaries} gives the necessary preliminaries. Section \ref{sec:sigma1} and  Section \ref{sec:sigma2} respectively treat those learning criteria whose characterization can be given in terms of $\sigmainf{1}$ and $\sigmainf{2}$ partial orders and how these learning criteria relate to ones already present in the literature. Section \ref{sec:conclusions} draws some further direction on these topics.

		\section{Preliminaries}
		\label{sec:preliminaries}

		\subsection{Sequences and structures}
		\label{prel:structures}
		
		We denote with $\cantor$ and $\baire$ respectively the set of all finite sequences of $\{0,1\}$ and the set of all finite sequences of natural numbers. The following definitions are given for elements of $\baire$ but similar definitions hold for elements of $\cantor$. Given $\sigma \in \baire$ we denote by $|\sigma|$ the length of $\sigma$: $|\cdot|$ is also used to denote the cardinality of a set. For $m < |\sigma|$, we denote the $m$-th element of $\sigma$ by $\sigma(m)$ and by $\sigma[m]$ the finite sequence having elements $\sigma(0),\dots,\sigma(m)$. The concatenation of two finite sequences $\sigma,\tau$ is denoted by $\sigma^\frown \tau$.
		We let $\str{i,j}$ denote a computable 1-1 mapping from $\nats \times \nats$ to $\nats$.

		The \emph{Cantor space} (denoted with $\Cantor$) and the \emph{Baire space}\ (denoted with $\Baire$) are represented as the collection of infinite binary sequences (respectively, of infinite sequence on natural numbers) equipped with the product topology of the discrete topology on $\{0,1\}$ ($\nats$). The forthcoming
		definitions are given for elements of $\Baire$ or $\baire$, but, similar definitions hold for elements of $\Cantor$ or $\cantor$. Given $p \in \Baire$ and $m \in \nats$, the definitions of $p(m)$ and $p[m]$ follow the ones for finite sequences. We denote by $i^j$ the sequence made of $j$-many $i$'s: in case $j=1$ we just write $i$ and we denote by $i^\nats$ the infinite sequence with constant value $i$. Given $p,q \in \Baire$, \emph{join} of $p$ and $q$, denoted by $p \oplus q$ is the element of $\Baire$ such that for any $i \in \nats$ $(p\oplus q)(2i):=p(i)$ and $(p\oplus q)(2i+1):=q(i)$. We will also consider elements of the product space $\nats^{\nats \times \nats}$: given $p \in \nats^{\nats \times \nats}$, we denote by $p^{[m]}$ the infinite sequence representing the $m$-th column of $p$, that is $p(\str{m,\cdot})$.

		Through this paper the structures we consider always have domain $\nats$ and are defined on a finite relational signature. For $s \in \nats$, we denote by $\esse\restriction_s$ the finite substructure of $\esse$ having domain $\{0,\dots,s\}$. Given two structures $\A$ and $\B$ we write $\A \hookrightarrow \B$ to denote that there exists an embedding of $\A$ in $\B$, and we say that $\A$ and $\B$ are \emph{copies} of each other if they are isomorphic. As often done in computable structure theory, we represent a structure $\A$ via its \emph{atomic diagram}, i.e., the collection of atomic formulas which are true in $\A$. Up to a suitable G{\"o}del numbering of the formulas in the language of $\A$, the atomic diagram of $\A$ can be considered as an element $p \in\Cantor$. In other words, we will have that $p(i)=1$ if the atomic formula having G{\"o}del number $i$ is satisfied by $\A$, and $p(i)=0$ otherwise. For more on computable structure theory we refer the reader to \cite{AK00,montalban2005beyond,Mon-Book}.
		In some cases, we need to work explicitly with a family of structures: most of the structures we consider are \emph{partial orders}, \emph{linear orders} and \emph{graphs}.
		
		A partial order $(P,\leq_P)$ is a structure having domain $P$ and a binary relation $\leq_P$ satisfying reflexivity, transitivity and antisymmetry. Given $a,b \in P$ we write $a <_P b$ to denote that $a \leq_P b$ and $b \nleq_P a$. A linear order is a partial order in which any two elements are comparable. Since no confusion should arise, we denote a partial order $(P, \leq_P)$ just by $P$.
		
		For a linear order $L$ we denote by $\leq_L$ the corresponding ordering relation of $L$ and given two elements $a,b \in L$ we write $a <_L b$ to denote that $a \leq_L b$ and $b \nleq_L a$. We use the following notations (some of them have already been used in the introduction): $\omega$ and $\omega^*$ are respectively the linear orders having order type of the natural numbers and of the negative integers, $\zeta$ is the linear order having order type of the integers, and for $n>1$, $L_n$ is the finite linear order with precisely $n$ elements.
		
		In our setting, requiring a structure to have domain $\nats$ is not a concern as long as it is infinite. On the other hand, for some of our proofs, it is convenient to consider finite linear orders. To avoid the problems of cardinality, we define the following operation:
		
		\begin{definition}
			\thlabel{definition:tilde}
			Given a partial order $L$, we define $\tilde{L}$ to be the partial order consisting of $L$ plus infinitely many pairwise incomparable elements which are also incomparable to the elements of $L$.
		\end{definition}
		
		We now give the necessary definitions for {graphs}. We will consider only graphs that are countable undirected and without self-loops. That is, a graph $G$ is a structure having domain a set of \emph{vertices} $V$ and binary relation $E$ satisfying anti-reflexivity and symmetry; a pair $(v, w) \in E$ is called an \emph{edge}. We will denote by $V(G)$ and $E(G)$ respectively the vertices and the edges of $G$. In this paper we assume all the graphs to  be countable, undirected and without self-loops. 
		
		We define the \emph{one-way infinite ray} (denoted by $\mathsf{R}$) and the \emph{infinite isolated graph} (denoted by $\mathsf{I}$) as the graphs having as vertices the natural numbers and edges respectively $\{(i,i+1): i \in \nats\}$ and $\emptyset$. We define the
		\begin{itemize}
			\item  \emph{ray of size} $n$, for $n \geq 2$, (denoted by $\mathsf{R}_n$) as the graph having $V(\mathsf{R}_n)=\{i : i < n\}$ and  $E(\mathsf{R}_n)=\{(i,i+1):i < n-1\}$;
			\item  \emph{cycle of size} $n$, for $n \geq 3$, (denoted by $\mathsf{C}_n$) as the graph having $V(\mathsf{C}_n)=\{i : i < n\}$ and $E(\mathsf{C}_n)=\{(i,i+1):i < n-1\} \cup \{(0,n-1)\}$;
			\item \emph{isolated graph of size} $n$, for $n \geq 0$ (denoted by $\mathsf{I}_n$) as the graph having $V(\mathsf{I}_n)=\{i: i<n\}$ and $E(\mathsf{I}_n)=\emptyset$.
		\end{itemize}
		
		Similarly to what we have done for finite linear orders, to avoid problems of cardinality when using finite graphs we introduce the following operation.
		Given two graphs $G_0$ and $G_1$ we define the \emph{disjoint union of $G_0$ and $G_1$} (denoted by $G_0 \oplus G_1$) so that
		\[V(G_0 \oplus G_1)=\bigcup_{i \leq 1}\{\langle i,v\rangle: v \in V(G_i)\} \text{ and }E(G_0 \oplus G_1)=\bigcup_{i \leq 1}\{(\langle i,v\rangle,\langle i,w\rangle) : (v,w) \in E(G_i) \}.\]
		Notice that, in case at least one between $G_0$ and $G_1$ has domain $\nats$, without loss of generality we can assume $G_0 \oplus G_1$ to have domain isomorphically mapped to $\nats$.

		\subsection{The learning paradigm}
		\label{sec:our_learning_paradigm}
		We now formally introduce the learning paradigm we are working with, and we define $\ex$-learnability. It is important to notice that in this framework, it is not specified how a family is enumerated; instead, we assume that any structure $\A$ is associated with a corresponding conjecture $\code{\A}$ that can be considered as a natural number. This conjecture essentially conveys the information \lq\lq this is $\A$\rq\rq. 
		From now on, when we write \lq\lq family of structures\rq\rq\ we are assuming the family to be countable and containing pairwise nonisomorphic countable structures.

		\begin{definition}
			\thlabel{definition:paradigm}
			Let $\K$ be a family of structures.
			\begin{itemize}
				\item The \emph{learning domain} ($\mathrm{LD}$)  is the collection of all isomorphic copies of the structures from $\K$. That is,
				$
				\ldofk\defas\bigcup_{\A \in \K} \{\esse : \esse\cong \A\}.
				$ Since any structure is represented with an element of $\Cantor$ (see Section \ref{prel:structures}), $\ldofk$ can be regarded as a subset of $\Cantor$.
				\item The \emph{hypothesis space} ($\mathrm{HS}$) contains, for each $\A\in \K$, a formal symbol $\code{\A}$  and a question mark symbol. That is, 
				$
				\mathrm{HS}(\K)\defas\{\code{\A} : \A \in  \K \}\cup \{?\}.
				$
				\item A \emph{learner} $\learnerM$ sees, by stages, finite portions of the atomic diagram of a given structure in the learning domain and is required to output conjectures. Thus, we mostly consider $\learnerM$ as a mapping from $\{\esse\restriction_s : \esse \in \ldofk\}$ to $\mathrm{HS}(\K)$. In some cases, since any structure is represented by an element of $\Cantor$ (see Section \ref{prel:structures}), we can formalize $\learnerM$ as a function from $2^{<\nats}$ to $\mathrm{HS}(\K)$.  
			\end{itemize}
		\end{definition}
		
		\begin{definition}
			Given a family of structures $\K$, we say that $\K$ is $\ex$-learnable if,  there exists a learner $\learnerM$ such that, on every $\esse\in\ldofk$, $\learnerM$  eventually stabilizes to a correct conjecture about its isomorphism type. That is, for every $\esse \in \ldofk$,
			$
			\lim_{n\to \infty} \learnerM(\esse\restriction_n)=\ulcorner \mathcal{A}\urcorner \text{ if and only if } \esse \cong \A.
			$

		\end{definition}
		
		As already said, explanatory learning is one of the most studied convergence criteria for a learner in classical algorithmic learning theory. Another important convergence criterion is \emph{behaviorally correct learning}, denoted by $\bc$. In the classical algorithmic learning theory, a learner $\learnerM$ learns an input function in behaviorally correct sense, if it produces indices for correct programs for the function (not necessarily  stabilizing on any one of them) for all but finitely many stages.
		In other words, $\bc$-learnability requires to converge \emph{semantically} while $\ex$-learnability asks to converge \emph{syntactically}. 
		To define $\bc$-learnability in our framework we need to adapt the hypothesis space of the paradigm in \thref{definition:paradigm}: that is, for a family $\K$, instead of considering $\mathsf{HS}(\K) \defas \{\code{\A}: \A \in \K\}$, we consider a new hypothesis space defined as $\{\code{\A, i}: i \in \nats \land \A \in \K\}$. Clearly, $\ex$-learnability implies $\bc$-learnability and, in the classical setting, it is a well-known result that $\bc$-learnability is more general than $\ex$-learnability (see \cite{CASE1983193}). On the other hand, in our framework, since our learners have no computational constraints, the two notions actually coincide. Indeed, the $\ex$-learner can immediately recognize whether two conjectures given by the $\bc$-learner refer to the same structure or not and consequently output the corresponding (unique) conjecture associated with the structure.

		Before introducing other learning paradigms, we give an example of a family that is $\ex$-learnable.
		Recall that $\omega$ and $\omega^*$ denote respectively the linear order isomorphic to the natural numbers and the linear order isomorphic to the negative integers. 
		\begin{proposition}
			\thlabel{proposition:omega_vs_omega*}
			$\{\omega,\omega^*\}$ is $\ex$-learnable.
		\end{proposition}
		\begin{proof}
			Given $\esse \in \ldofk$, we define $min_s\defas \min_{\leq_{\esse}}\{n: n \in \esse\restriction_s\}$ and $max_s\defas \max_{\leq_{\esse}}\{n: n \in \esse\restriction_s\}$, and we let $c_{min}(s)\defas |\{t \leq s : min_s=min_{t}\}|$ and $c_{max}(s)\defas |\{t \leq s : max_s=max_{t}\}|$. We define a learner $\learnerM$ as follows:
			\[\learnerM(\esse\restriction_s) \defas 
			\begin{cases}
				\code{\omega} & \text{if } c_{min}(s)> c_{max}(s),\\
				\code{\omega^*} & \text{if } c_{min}(s)< c_{max}(s),\\
				? & \text{otherwise}.
			\end{cases}
			\]
			Suppose that $\esse \cong \omega$ (the case for $\esse \cong \omega^*$ is analogous): then there will be a stage $s$ such that $n \in \esse\restriction_s$ and $n$ is such that $n = \min_{\leq_{\mathcal{S}}}\{ m : m \in \esse\}$. Hence, there exists a stage $s' \geq s$ such that for all $t \geq s'$, $c_{min}(t)> c_{max}(t)$. By the first case of $\learnerM$'s definition, this means that for all $t \geq s'$, $\learnerM(\esse\restriction_t)= \code{\omega}$, i.e., $\K$ is $\ex$-learnable.
		\end{proof}
		
		In \cite{bazcipsan_cie}, the authors explored $\alpha$-\emph{learnability}, that can be introduced as a restriction of $\ex$-learnability. To define it, we first need the following definition.
		\begin{definition}
			Let $\learnerM$ be a learner, $\K$ be a family of structures, and let $\function{c}{\{\esse\restriction_s : \esse \in \ldofk\}}{ \mathrm{Ordinals}}$. We say that  $c$ is a \emph{mind change counter for} $\learnerM$  if, for any $s \in \nats$,
			
			\begin{itemize}
				\item $c(\esse\restriction_{s+1}) \leq c(\esse\restriction_s)$, and
				
				\item $c(\esse\restriction_{s+1})<c(\esse\restriction_s)$ if and only if $\learnerM$ changes its mind at  $\esse\restriction_{s+1}$.
			\end{itemize}
			We say that $\K$ is $\alpha$-learnable if there exists a learner $\learnerM$ that $\ex$-learns $\K$ and, for every $\esse\in\ldofk$, stabilizes to the correct conjecture making at most $\alpha$-many mind changes. That is there is a mind change counter $c$ for $\learnerM$ and $\K$  such that $c(\str{})=\alpha$. We say that $\K$ is \emph{properly $\alpha$-learnable} if $\K$ is $\alpha$-learnable but not $\beta$-learnable for all $\beta<\alpha$.
		\end{definition}
		Notice that, with our definition, any learner has an associated counter. A priori, one could have defined different counters that are in a certain sense not \lq\lq optimal\rq\rq\ with respect to $\learnerM$’s mind changes, e.g., counters that decrease even if $\learnerM$ did not change its mind at $\sigma$. 
		
		As we have mentioned in the introduction $0$-learnability (i.e., learning with no mind-changes) corresponds to what in classical algorithmic learning theory is called $\Fin$-learnability.   
		
		To assess the power of various learning criteria and of the equivalence relations introduced in the next section, it will be fundamental to analyze the logical complexity of the formulas needed to separate structures in the family to be learned. Such complexity will be measured with respect to the infinitary logic $\mathcal{L}_{\omega_1\omega}$ which allows to take countable conjunctions and disjunctions. 
		
		\subsection{Infinitary logic}
		\label{sec:infinitary}
		The following definitions (and much more) can be found e.g., in \cite{montalban2005beyond}.
		\label{sec:infinitarylogic}
		Given a language $L$, $\Lomega$ is defined as the smallest class such that:
		\begin{itemize}
			\item all finitary quantifier-free $L$-formulas are in $\Lomega$;
			\item if $\varphi$ is already in $\Lomega$, then so are all formulas $\forall x \varphi$, $\exists x \varphi$;
			\item if $\overline{x}$ is a finite tuple of variables and $S \subseteq \Lomega$ is a countable set of formulas whose free variables are contained in $\overline{x}$ then both 
			\begin{itemize}
				\item the infinitary disjunction of the formulas in $S$ denoted by $\bigdoublevee
				_{\varphi \in S} \varphi$,
				\item the infinitary conjunction of the formulas in $S$ denoted by $\bigdoublewedge
				_{\varphi \in S} \varphi$,
			\end{itemize}
			are in $\Lomega$.
		\end{itemize}
		
		The complexity of the formulas is defined similarly to first-order logic counting the alternation of quantifiers: here, infinitary disjunctions and conjunctions are treated respectively as existential and universal quantifiers. 
		\begin{definition}
			Fix a countable language $L$. For every $\alpha<\omega_1$ we define the sets 
			$\sigmainf{\alpha}$ and $\piinf{\alpha}$ of $L$-formulas inductively.\\
			For $\alpha=0$ the $\sigmainf{\alpha}$ and $\piinf{\alpha}$ formulas are the quantifier-free first-order $L$-formulas.
			\begin{itemize}
				\item A $\sigmainf{\alpha}$ formula $\varphi(\overline{x})$ is the countable disjunction
				\[ \bigdoublevee_{i \in I} (\exists \overline{y}_i)( \psi_i(\overline{x},\overline{y}_i))\]
				where $I$ is a countable set and each $\psi_i$ is a $\piinf{\beta_i}$ formula for $\beta_i < \alpha$;
				\item A $\piinf{\alpha}$ formula $\varphi(\overline{x})$ is the countable conjunction
				\[ \bigdoublewedge_{i \in I} (\forall \overline{y}_i)( \psi_i(\overline{x},\overline{y}_i))\]
				where $I$ is a countable set and each $\psi_i$ is a $\sigmainf{\beta_i}$ formula for $\beta_i < \alpha$.
			\end{itemize}
		\end{definition}

		Given a structure $\A$ we define the $\sigmainf{n}$\emph{-theory} of $\A$ as $\thsigma{n}{\A}\defas\{\varphi \in \Lomega: \A \models \varphi \text{, and $\varphi$ is a closed } \sigmainf{n}\text{ formula}\}$.

		\begin{definition} 
			\thlabel{definition:sigmainfnposet}
			Let $\K=\{\A_i : i \in \nats\}$ be a family of countable structures. Then, 
			\begin{itemize}
				\item $\K$ is a \emph{$\sigmainf{\alpha}$\text{-}antichain}, if the partial order
				$
				\big(\{ \thsigma{\alpha}{\A_i} : i\in\nats \},\subseteq\big)\text{ is an antichain}
				$, i.e., for any $i \neq j$ there exists two $\sigmainf{\alpha}$ formulas $\varphi_i,\varphi_j$ such that $\A_i \models \varphi_i \land \A_i \not\models \varphi_j$ and $\A_j \models \varphi_j \land \A_j \not\models \varphi_i$;
				\item $\K$ is a \emph{$\sigmainf{\alpha}$\text{-}strong antichain} if there are $\sigmainf{\alpha}$ formulas $\{\varphi_i : i \in \nats \}$ so that
				$
				\A_i\models \varphi_j \Leftrightarrow i=j.
				$
				\item $\K$ is a \emph{$\sigmainf{\alpha}$\text{-}partial order} if, for all $i\neq j\in\nats$, 
				$
				\thsigma{\alpha}{\A_i} \neq  \thsigma{\alpha}{\A_j};
				$
			\end{itemize} 
		\end{definition}
		The following Lemma (without proof) states the obvious relations between the three definitions above.
		\begin{lemma}
			\thlabel{lemma:easyposet}
			The following hold:
			\begin{itemize}
				\item[(a)]  If $\K$ is a $\sigmainf{\alpha}$-strong antichain then $\K$ is also a $\sigmainf{\alpha}$-antichain.
				\item[(b)]  If $\K$ is a $\sigmainf{\alpha}$-antichain then $\K$ is also a $\sigmainf{\alpha}$-partial order.
				\item[(c)] if $\K$ is finite, then $\K$ is a $\sigmainf{\alpha}$-strong antichain if and only if $\K$ is a $\sigmainf{\alpha}$-antichain.
			\end{itemize}
		\end{lemma}
		
		With this definition in mind, we can reformulate  \cite[Theorem 3.1]{bazhenov2020learning} mentioned in the introduction.
		\begin{theorem}
			A family $\K$ is $\ex$-learnable if and only if $\K$ is $\sigmainf{2}$-strong antichain.
		\end{theorem}

		\subsection{$E$-learnability}
		\label{sec:elearnability}
		As already mentioned in the introduction, the definition of $E$-learnability borrows ideas from descriptive set theory. 
		One of the main themes of this subject is the study of the complexity of equivalence relations and, a popoular way to evaluate this complexity is via \emph{reductions}.
		
		In general, a {reduction} from an equivalence relation $E$ on a space $X$
		to an equivalence relation $F$ on $Y$ is a function $\Gamma : X \rightarrow Y$ such that $x E x'$
		if and only if $\Gamma(x) F \Gamma(x')$. In this paper, we will always use \emph{continuous} reductions, i.e., we assume that the reduction $\Gamma$ is a continuous function. This is indeed a natural choice, as we want $\Gamma$ to
		mimic the behavior of a learner. Therefore, as the learner outputs a conjecture based on
		an initial segment of the atomic diagram of the given structure, we expect a finite portion
		of the output of $\Gamma$ on a given element $x$ to be determined by an initial segment of $x$. We
		make this observation formal. To do so we define what is a \emph{Turing operator}. Consider a partial computable function $\Phi$ with oracle $X$ mapping finite sequences to finite sequences such that $\sigma \sqsubseteq \tau$ implies $\Phi^X(\sigma) \sqsubseteq \Phi^X(\tau)$. Then, $\Phi$ is considered as a (partial) Turing operator with oracle $X$, mapping $\Baire$ to $\Baire$ by mapping infinite sequence $p$ to $\bigcup_s \Phi^X(p[s])$, if this is infinite.
		
		During the paper, we will often use the following well-known fact (we
		refer the reader to \cite[Lemmma 2.1]{bazhenov2021learning} for a short proof of this).
		
		\begin{lemma}[Folklore]
			\thlabel{lemma:folklore}
			If $\Gamma:\Baire \rightarrow \Baire$ is continuous, then there is an oracle $X$ and a Turing operator $\Phi$ such that $\Gamma=\Phi^{X}$.
		\end{lemma}
		Clearly, the same Lemma holds if, instead of $\Baire$, the domain and/or range of $\Gamma$ and the Turing
		operator $\Phi$ is defined in spaces like $\Cantor$, $\Baire$ and $\nats^{\nats \times \nats}$.

		In the introduction, we have already mentioned that $\ex$-learnability and $E_0$-learnability coincide and that replacing $E_0$ with other equivalence relations of lower/higher complexity unlocks a natural way to stratify learning problems. Now we make this formal with the following two definitions.
		
		\begin{definition}
			\thlabel{definition:elearnability}
			A family of structures $\K$ is $E$-learnable if there is a function $\Gamma: \ldofk \rightarrow \Baire$
			which continuously reduces $\ldofk$ to $E$.
		\end{definition}
		
		\begin{remark}
			\thlabel{remark:elearning}
			\thref{definition:elearnability} differs from \cite[Definition 3.2]{bazhenov2021learning} where $\Gamma$ is a total function from $\Cantor$ to $\Cantor$ continuously reducing $\ldofk$ to $E$. First of all, notice that the all the (non) reductions and characterizations in \cite{bazhenov2021learning} hold the same regardless of which definition we are using. On the other hand, we point out two mistakes in \cite{bazcipsan_cie}: here the authors consider $E$-learnability as defined in \cite[Definition 3.2]{bazhenov2021learning}. The mistakes are in \cite[Proposition 1, Theorem 3]{bazcipsan_cie}: namely these results do not hold with their definition but they do with ours. We provide a counterexample to  \cite[Proposition 1]{bazcipsan_cie}: since \cite[Theorem 3]{bazcipsan_cie} is a generalization of the latter, the same counterexample works for both results. \cite[Proposition 1]{bazcipsan_cie} states that $\mathbf{Fin}$-learnability implies $Id$-learnability. As already mentioned, the claim holds with our definition of $E$-learnability. Let  $\K=\{\mathsf{C}_3 \oplus \mathsf{I},\mathsf{C}_4 \oplus \mathsf{I}\}$. We claim that $\K$ is $\mathbf{Fin}$-learnable but not $\Id$-learnable (where $\Id$-learnability here is defined as in \cite[Definition 3.2]{bazhenov2021learning}). The fact that $\K$ is $\mathbf{Fin}$-learnable will follow from \thref{theorem:sigma1strongantichain}, but it is also immediate just by defining a learner as follows. Given some structure $\esse$ in input, the learner outputs \lq\lq ?\rq\rq\ unless there is a first stage in which either a copy of $\mathsf{C}_3$ or a copy of $\mathsf{C}_4$ appears in $\esse$: if so, the learner outputs $\code{\mathsf{C}_3 \oplus \mathsf{I}}$ or $\code{\mathsf{C}_4 \oplus \mathsf{I}}$ accordingly. In case $\esse \notin \ldofk$, then the learner produces either a constant sequence of \lq\lq ?\rq\rq\ or an eventually constant $\code{\mathsf{C}_i \oplus \mathsf{I}}$ sequence for $i \in \{3,4\}$. To show that $\K$ is not $\Id$-learnable (where $\Id$-learnability is defined as in \cite[Definition 3.2]{bazhenov2021learning}), we proceed as follows. First notice that $\Gamma(\mathsf{C}_3 \oplus \mathsf{I})=p$ and $\Gamma(\mathsf{C}_4 \oplus \mathsf{I})=q$ for some $p, q \in \Cantor$ where $p\ \centernot{\Id}\ q$. Hence we start defining a copy $\esse$ that is made only of isolated vertices. Since $\Gamma$ is total, it needs to be defined also if $\esse \notin \ldofk$. Then it suffices to wait for a stage $s$ such that $\Gamma(\esse)[s] \neq p[s]$ or $\Gamma(\esse)[s] \neq q[s]$. Without loss of generality, assume $\Gamma(\esse)[s] \neq p[s]$: then it suffices to build $\esse$ so that $\esse \cong \mathsf{C}_3 \oplus \mathsf{I}$ and this proves the claim.
			
			On the other hand, the reason for which  \cite[Proposition 1, Theorem 3]{bazcipsan_cie} do not work is \lq\lq not natural\rq\rq\ in the following sense. The notion of $E$-learnability generalizes $\ex$-learning (see e.g., \cite[Theorem 3.1]{bazhenov2020learning}). In $\ex$-learning (and in the other variants like $\mathbf{Fin}$-learning) we never considered the case in which the input could have been a structure outside the learning domain of the target family and so we find more natural that the $\Gamma$ in \thref{definition:elearnability} should be defined only on $\ldofk$ and not in the whole space of structures. 
			
			This remark hopefully motivated that the definition we are using is more natural, and we highlight once again that the notion of $E$-learnability we are using in this paper is the one defined in \thref{definition:elearnability}.
		\end{remark}

		The following notion of reducibility between equivalence relations captures their learning-theoretic strength of $E$-learnability. Notice that it makes sense to consider this notion of reducibility also in the context of learning of finite families. Indeed, there are families consisting of $2$ structures that are not $\ex$-learnable: This is in contrast with classical algorithmic learning theory, where, for example, finite families of computable functions are always $\ex$-learnable. Regarding the computational power needed to learn such families, the authors in \cite{bazhenov2021turing} showed that there exists a family of $2$ structures that is $\ex$-learnable but not by a computable learner.
		\begin{definition}
			Let $E$ and $F$ be equivalence relations: $E$ is \emph{learn-reducible}\ to $F$ (in symbols $E \learnreducible{} F$),
			if every $E$-learnable family of structures is also $F$-learnable. Similarly, $E$ is \emph{finitary learn-reducible}\ to $F$ (in symbols $E \learnreducible{\finitary} F$), if every finite $E$-learnable family is also 
			$F$-learnable. Let also
			\begin{itemize}
				\item $E\ \learnequiv{}\ F$ if and only if $E \learnreducible{} F$ and $F \learnreducible{} E$, and
				\item $E\ \strictlearnincomparable{}\ F$ if and only if $E \notlearnreducible{} F$ and $F \notlearnreducible{} E$.
			\end{itemize}
			Similar notations are used for the finitary learn-reducibility.
		\end{definition}
		Notice that in case we want to compare $E$-learnability for some equivalence relation $E$ with some learning paradigm $\mathbf{X}$ (e.g., like in \cite[Theorem 3.1]{bazhenov2021learning}) we will use the same notation. For example, we write $E_0 \learnreducible{} \ex$ meaning \lq\lq Any $E_0$-learnable family is also $\ex$-learnable\rq\rq.

		We now introduce the equivalence relations we consider in this paper. In the context of learnability for algebraic structures, with the exception of $=_\nats$ and $\erange$, these have already been studied in \cite{bazhenov2021learning}.
		
		\begin{enumerate}[(i)]
			\item For $n,m \in \nats$, $n \ =_\nats \ m$ if and only if $n=m$.
			\setcounter{newcounter}{\value{enumi}}
			
		\end{enumerate}
		For $p,q \in \Baire$, 
		\begin{enumerate}[(i)]
			\setcounter{enumi}{\value{newcounter}}
			
			\item   $p \ \Id \ q$ if and only if $(\forall n)(p (n)=q(n)).
			$
			\item
			$p \ E_0 \ q $ if and only if $(\exists m)(\forall n\geq m)(p (n)=q(n)).
			$
			\item  $p\ \erange\ q$ if and only if 
			$
			\{ p(m)\,\colon m\in\nats \} = \{ q(m)\,\colon m\in\nats\}.
			$
			\setcounter{newcounter}{\value{enumi}}
		\end{enumerate}
		For $p,q \in \nats^{\nats \times \nats}$
		\begin{enumerate}[(i)]
			\setcounter{enumi}{\value{newcounter}}
			\item $p \ E_3\ q$ if and only if 
			$
			(\forall m )( p^{[m]} \ E_0 \ q^{[m]}).
			$
			\item $p\ \eset\ q$  if and only if $\{ p^{[m]} : m \in \nats \} = \{q^{[m]} : m \in \nats\}$.
			
		\end{enumerate}
		
		Notice that the equivalence relations studied in \cite{bazhenov2021learning} were defined on $\Cantor$ and $2^{\nats \times \nats}$ rather than $\Baire$ and $\nats^{\nats\times \nats}$, but the next Lemma shows that the two versions coincide. For one of the equivalence relations $E$ in \textbf{(ii)}-\textbf{(vi)} (except \textbf{(iv)}), let us denote by $E(\Cantor)$ the corresponding equivalence relations with domain $\Cantor$ or $2^{\nats \times \nats}$ considered in \cite{bazhenov2021learning}.
		
		\begin{lemma}
			For any equivalence realtion $E$ in $\mathbf{(ii)}-\mathbf{(vi)}$ (except $\mathbf{(iv)}$), $E \learnequiv{} E(\Cantor)$.
		\end{lemma}
		\begin{proof}
			
			The fact that $E(\Cantor) \learnreducible{} E$ is trivial.
			
			To show that $E_0 \learnequiv{} E_0(\Cantor)$ just notice that the reduction given in \cite[Proposition 6.1.2]{gao2008invariant} is continuous. Notice that the same reduction, applied to every $p^{[m]}$ of an element $p \in \nats^{\nats \times \nats}$ also shows that $E_3 \learnequiv{} E_3(\Cantor)$.
			
			To show that $\Id \learnequiv{} \Id(\Cantor)$ it is easy to notice that the function $f:\Baire \rightarrow \Cantor$ defined as $f(p):=0^{p(0)}10^{p(1)}1\dots10^{p(n)}1\dots$ is a continuous reduction from $\Id$ to $\Id(\Cantor)$. To conclude notice that the same reduction, applied to every $p^{[m]}$ of an element $p \in \nats^{\nats \times \nats}$ also shows that $E_{set} \learnequiv{} E_{set}(\Cantor)$.
		\end{proof}
		
		We give an intuitive idea of the equivalence relations defined above.  The first two are respectively the identity on the natural numbers and the identity on infinite sequences. Instead, $E_0$ relaxes this notion by requiring the two infinite sequences to be the same from a certain point on. The equivalence relation $E_{range}$ considers two infinite sequences of natural numbers to be the same if they have the same range. Both $E_3$ and $\eset$ are defined on columns of infinite sequences. The equivalence relation $E_3$ demands that columns with the same index are the same from some point on; $\eset$ instead does not care neither of the indices of the columns nor of the multiplicity of each column. The only requirement is that a column appearing in the first set of infinite sequences must appear also in the second one and vice versa. 
		
		The following theorem collects some of the results obtained in \cite{bazhenov2021learning} that will be useful in the rest of the paper.

		\begin{theorem} 
			\thlabel{theorem:prevwork}
			The following relations hold between $\Id$, $E_0$, $E_3$, and $\eset$: 
			\[\Id\strictlearnreducible{} E_0 \strictlearnreducible{} E_3 \strictlearnreducible{} \eset \text{ and } \Id\strictlearnreducible{\finitary} E_0 \learnequiv{\finitary} E_3  \strictlearnreducible{\finitary} \eset.\]
		\end{theorem}

		We mention that  \cite{bazhenov2021learning}  also considered the learning power of the equivalence relations $E_1$, $E_2$, and $Z_0$: on the other hand, since they are not important for the purpose of this paper, we omit their definitions. We just mention here that both $E_1$ and $E_2$ have the same learning power as $E_0$ (\cite[Theorems 4.1 and 4.2]{bazhenov2021learning}), while $Z_0$ has the same learning power as $E_0$ for finite families (\cite[Theorem 6.2]{bazhenov2021learning}). From classical descriptive set-theoretic results, it also follows that $Z_0$ lies above $E_3$ and strictly below $\eset$ for infinite families, but it is still open whether $Z_0 \learnequiv{} E_3$.

		\section{Characterizing learnabilities in terms of $\sigmainf{1}$-formulas}
		\label{sec:sigma1}
		In this section we consider learning paradigms whose characterizations can be given in terms of $\sigmainf{1}$ formulas. 
		\subsection{$\sigmainf{1}$-(strong) antichains}
		
		It turns out that the analogous of $\ex$-learnability, i.e., the learning paradigm that can be characterized in terms of $\sigmainf{1}$-strong antichains is $\Fin$-learnability.

		\begin{theorem}
			\thlabel{theorem:sigma1strongantichain}
			$\K$ is $\Fin$-learnable if and only if $\K$ is a  $\sigmainf{1}$-strong antichain.
		\end{theorem}
		
		\begin{proof}
			Assume that $\K$ is a $\sigmainf{1}$ strong antichain. For any $i$, let $\varphi_i$ be the $\sigmainf{1}$ formula such that for any $j$, $\A_j \models  \varphi_i \iff i=j$. 
			
			Given $\esse \in \ldofk$, since any $\varphi_i$ is a $\sigmainf{1}$ formula, we may wait for a stage $s$ such that $\esse\restriction_s \models \varphi_i$ for some $i$ (notice that such an $i$ is unique). Then, for any $t <s$, let $\learnerM(\esse\restriction_t)=?$ and for any $t \geq s$, let $\learnerM(\esse\restriction_t)=\code{\A_i}$: it is clear that $\learnerM$ $\Fin$-learns $\K$.
			
			For the opposite direction, suppose that $\K$ is $\Fin$-learnable by some learner $\learnerM$. Then for any $i$, consider $\learnerM(\A_i)$ and let $s_i\defas \min \{s: \learnerM(\A_i\restriction_{s}) \neq ?\}$. Let  $\varphi_i$ be the $\sigmainf{1}$ formula saying that there exists a finite substructure isomorphic to $\A_i\restriction_{s_i}$. We claim that for any $i,j$, $\A_j \models  \varphi_i \iff i=j$. 
			
			Suppose otherwise, i.e., suppose that there is some $j \neq i$ such that $\A_j \models \varphi_i$.
			This means that we could define a copy $\esse$ of $\A_j$ such that $\esse\restriction_{s_i} \cong \A_i\restriction_{s_i}$. But then $\learnerM(\esse\restriction_{s_i})=\code{\A_i}$ while $\esse \cong \A_j$ obtaining that $\learnerM$ does not $\Fin$-learn $\K$.
		\end{proof}
		
		It is easy to show that $\Fin$-learnability can also be interpreted in terms of $=_\nats$-learnability.
		
		\begin{theorem}
			\thlabel{theorem:fincharacterization}
			Let $\K=\{\A_i:i \in \nats\}$: $\K$ is $\Fin$-learnable if and only if $\K$ is $=_\nats$-learnable.
		\end{theorem}
		\begin{proof}
			Suppose $\K$ is $\Fin$-learnable. Given $\esse \in \ldofk$ we define a continuous reduction $\Gamma$ from $\ldofk$ to $=$ just letting $\Gamma(\esse)=\learnerM(\esse\restriction_s)(s)$ where $s=\min\{t: \learnerM(\esse\restriction_t) \neq ?\}$.
			
			For the opposite direction, given a continuous reduction $\Gamma$ from $\ldofk$ to $=_\nats$, let $a_i=\Gamma(A_i)$. Then, given $\esse \in \ldofk$, let $\learnerM(\esse\restriction_s)=?$ if $\Gamma(\esse\restriction_s)$ is not defined. Since $\Gamma$ is a reduction from $\ldofk$ to $=$ there will be a stage $t$ such that $\Gamma(\esse\restriction_t)=a_i$ for some $i \in \nats$: at this point, for any $k\geq t$ let $\learnerM(\esse\restriction_k)= \code{\A_i}$. 
		\end{proof}

		The following learning criterion was firstly introduced in the context of learning of total recursive functions by Freivalds, Karpinski and Smith in \cite{colearning}.
		
		\begin{definition}
			\thlabel{definition:colearning}
			We say that $\K:=\{\A_i:i \in \nats\}$ is \emph{$\col$-learnable} if there is a learner $\learnerM$ such that for any $\esse \in \ldofk$, \[\{\learnerM(\esse\restriction_s) : s \in \nats\}=\{\code{\A_i}:i \in \nats\} \cup \{?\} \setminus \{\code{\A_j}\} \iff \esse=\A_j.\]
		\end{definition}
		The next proposition shows that $\col$-learnability can be characterized in terms of $\sigmainf{1}$-antichains.
		
		\begin{theorem}
			\thlabel{theorem:12antichains}
			
			$\K$ is $\col$-learnable if and only if $\K$ is a $\sigmainf{1}$-antichain.
		\end{theorem}
		\begin{proof}
			For the left-to-right direction, we prove the contrapositive. Suppose that $\K\supseteq\{\A,\B\}$ is not a $\sigmainf{1}$-antichain but there is a learner $\learnerM$ which $\col$-learns $\K$. Without loss of generality assume that $\thsigma{1}{\A}\subset \thsigma{1}{\B}$ and let $\varphi$ be  such that $\B \models \varphi$ and $\A \not\models \varphi$.
			Furthermore, we can assume that if $\esse \cong\A$ then $\code{\A} \notin \{\learnerM(\esse\restriction_t): t \in \nats\}$, while if $\esse\cong\B$ then $\code{\B}\notin \{\learnerM(\esse\restriction_t): t \in \nats\}$. We define a copy $\esse$ as follows: At the beginning we start defining $\esse$ as a copy of $\A$: if for all $s$, $\code{\B} \notin \{\learnerM(\esse\restriction_t): t \leq s\}$ we get that $\learnerM$ fails to $\col$-learn $\K$. Otherwise, assume that at stage $s$, $\code{\B} \in \{\learnerM(\esse\restriction_t): t \leq s\}$. Then we let $\esse' \cong \B$ be an extension
			of $\esse\restriction_t$, so that $\learnerM$ does not $\col$-learn $\esse'$.
			
			For the opposite direction, let $\K=\{\A_i:i \in \nats\}$ and, for any $i \neq j$, let $\varphi_{ij}$ be a $\sigmainf{1}$ formula that is satisfied by $\A_i$ but not by $\A_j$. We say that ${\A_n}$ is \emph{triggered} at stage $s$ if there exists some $j \neq n$ such that $\esse\restriction_s\models \varphi_{jn}$. Given $\esse \in \ldofk$, let 
			\[\learnerM(\esse)(\str{i,s}):=\begin{cases}
				\code{\A_i}& \text{if } {\A_i}  \text{ is triggered at stage } s,\\
				? & \text{otherwise.}
			\end{cases}\]
			To conclude the proof, notice that if $\esse \cong \A_i$, then ${\A_i}$ is never triggered and this implies  that $\code{\A_i} \notin \{\learnerM(\esse\restriction_s):s \in \nats\}$. If $\esse\cong \A_j\not\cong\A_i$, then there is some stage $s$ such that ${\A_i}$ is triggered, i.e., the stage such that $\esse\restriction_s \models \varphi_{ji}$ and hence $\code{\A_i} \in \{\learnerM(\esse\restriction_s):s \in \nats\}$. This concludes the proof.
		\end{proof}
		
		We conclude this section summarizing the relationships between $\Fin$-learnability, $\col$-learnability and another natural learning criterion, namely $\Id$-learnability, already considered in \cite{bazhenov2021learning,bazcipsan_cie}.

		\begin{corollary}
			\thlabel{corollary:finidco}
			$\Fin\strictlearnreducible{}\Id \strictlearnreducible{} \col$ and
			$\Fin\learnequiv{\mathrm{fin}} \Id \learnequiv{\mathrm{fin}} \col$.
		\end{corollary}
		\begin{proof}
			To prove that  $\Fin\learnreducible{} \Id$ (and, in particular, $\Fin\learnreducible{\mathrm{fin}} \Id$) notice that $\Fin \learnequiv{} =_\nats$ (\thref{theorem:fincharacterization}) and it is easy to notice that $=_\nats \learnreducible{} \Id$. The fact that $\Id \not\learnreducible{} \Fin$
			follows from \cite[Theorem 1, Proposition 1]{bazcipsan_cie}: to provide an example of a family that is $\Id$-learnable but not $\Fin$-learnable,
			consider the family of undirected graphs $\K=\{\mathsf{C}_n \oplus \mathsf{I}: n \geq 3\} \cup \{\mathsf{R} \oplus \mathsf{I}\}$.
			
			To prove that $\Id\learnreducible{}\col$, suppose that $\K=\{\A_i:i \in \nats\}$ is $\Id$-learnable and let $\Gamma$ be the reduction witnessing that $\ldofk$ continuously reduces to $\Id$. We define a learner $\learnerM$ which co-learns $\K$ as follows: Let $\learnerM(\esse\restriction_s):=$
			\[
			\begin{cases}
				\code{\A_k}:=\min\{\code{\A_i} : \Gamma(\esse\restriction_s) \text{ is inconsistent with } \Gamma(\A_i\restriction_s)  \text{ and } & \text{ if $\code{\A_k}$ exists,}\\
				 \qquad \qquad \qquad  (\forall t < s)(\learnerM(\esse\restriction_t) \neq \code{\A_i})\}\\
				? & \text{otherwise.}
			\end{cases}
			\]   
			Intuitively, $\learnerM$ outputs $\code{\A_i}$ if it witnesses that the $\Gamma(\esse) \centernot{\Id} \Gamma(\A_i)$.
			To see that $\learnerM$ $\col$-learns $\K$ given $\esse \in \ldofk$, assume that $\esse \cong \A_i$. Notice that  for every $t$, $\learnerM(\esse\restriction_t) \neq \code{\A_i}$: otherwise, by $\learnerM$'s definition, there would be a stage $s$ such that $\Gamma(\esse)[s] \neq \Gamma(\A_i)[s]$, contradicting the assumption that $\K$ is $\Id$-learnable. To conclude the proof, we need to show that for any $j \neq i$ there exists some $t$ such that $\learnerM(\esse\restriction_t)=\code{\A_j}$. To prove this, let $s_j:=\min\{t: \Gamma(\esse\restriction_t)\text{ is inconsistent with } \Gamma(\A_j\restriction_t)\}$ (notice that such an $s_j$ must exists since $\K$ is $\Id$-learnable). Notice that there exists some $t\leq s_j+j+1$ such that $\learnerM(\esse\restriction_{{s_j+j+1}})=\code{\A_j}$: this concludes the proof of $\Id \learnreducible{} \col$.
		
			To prove that $\col \not\learnreducible{} \Id$, for $n,m \geq 3$, let $\A_n= \underset{\begin{subarray}{c} n,m \geq 3\\ n \neq m \end{subarray}}{\bigsqcup} \mathsf{C}_m$, i.e., the disjoint union of all cycles of length $m$ for $m \neq n$. Consider $\K=\{\A_n: n \geq 3\}$. It is easy to notice that $\K$ is $\col$-learnable: indeed, we can define a learner $\learnerM$ that, given in input $\esse \in \ldofk$, outputs $\code{\A_n}$ when $\esse$ contains a copy of $\mathsf{C}_n$. Clearly, if $\esse \cong \A_n$, $\mathsf{C}_n$ will never be contained in $\esse$ and hence $\learnerM$ will never output $\code{\A_n}$. Since every $\mathsf{C}_m$ with $m \neq n$ and $m \geq 3$ will be contained in $\esse$, $\learnerM$ will eventually output $\code{\A_m}$ for every $m \neq n$ and $m\geq 3$, showing that $\K$ is $\col$-learnable. It remains to show that $\K$ is not $\Id$-learnable: Assume that there exists a continous reduction $\Gamma$ from $\ldofk$ to $\Id$. Take any two structures $\A_n,\A_m \in \K$ with $n \neq m$ and $n,m \geq 3$ and let $s := \min \{t: \Gamma(\A_n\restriction_t) \text{ is inconsistent with } \Gamma(\A_m\restriction_t)\}$ (such an $s$ must exists as we are assuming that $\K$ is $\Id$-learnable). It is easy to check that for any $n,m \geq$ there exists some $k \notin \{0,1,2,n,m\}$ such that $\A_k$ contains both $\A_n\restriction_s$ and $\A_m\restriction_s$: this implies the existence of two different copies of $\esse,\esse' \cong \A_k$ where $\esse$ starts as $\A_n\restriction_s$ and $\esse'$ starts as $\A_m\restriction_s$,  such that $\Gamma(\esse\restriction_s)$ is inconsistent with $\Gamma(\esse'\restriction_s)$, a contradiction.

			To conclude the proof of the proposition it suffices to show that $\col \learnreducible{\mathrm{fin}} \Fin$: by \thref{theorem:fincharacterization} this is equivalent to prove $\col \learnreducible{\mathrm{fin}} \equalitynats$.  Suppose $\K=\{\A_i : i \leq n\}$ is co-learnable by a learner $\learnerM$. We define a continuous reduction $\Gamma$ from $\ldofk$ to $=_\nats$ letting, for any $\esse \in \ldofk$ and $i \leq n$, $\Gamma(\esse)= \code{\A_i}$ if and only if for any $j\leq n$ such that $j \neq i$, $\code{\A_j} \in \{\learnerM(\esse\restriction_s):s \in \nats\}$.
		\end{proof}
		
		Among the learning criteria we have considered in this section the one that is missing a syntactic characterization is $\Id$-learnability. Such a characterization was given \cite{bazhenov2024learning} in a slightly different form, as the author considers $E$-learnability as defined in \cite[Definition 3.2]{bazhenov2021learning}.
		\begin{theorem}{\cite[Theorem 4]{bazhenov2024learning}}
			\thlabel{theorem:idcharacterization}
			Let $\K=\{\A_i : i \in \nats\}$: $\K$ is $\Id$-learnable if and only if there is a family of $\sigmainf{1}$-formulas $\Theta=\{\varphi_j : j \in \nats\}$ so that:
			\begin{itemize}
				\item for every $\varphi \in \Theta$ there exists some $\psi \in \Theta$ such that for every $i$, $\A_i \models (\varphi \iff \lnot \psi)$;
				\item if $i \neq j$, then there exists some $\varphi \in \Theta$ such that $\A_i \models \varphi$ and $\A_j \models \lnot \varphi$.
			\end{itemize}
		\end{theorem}

		\subsection{$\sigmainf{1}$-partial orders}
		So far the notions of learnabilities we considered are linearly ordered with respect to learn-reducibility. As announced in the introduction, we give the first example of a natural equivalence relation such that its associated learnability notion is incomparable with respect to learn-reducibility with $\ex$-learning.

		\begin{theorem}
			\thlabel{theorem:characterization_Erange}
			A family $\K$ is $\erange$-learnable if and only if $\K$ is a $\sigmainf{1}$-partial order.
		\end{theorem}
		\begin{proof}
			Let $\K= \{\A_i:i \in \nats\}$.  For the left-to-right direction, we prove the contrapositive. Suppose that there are $i\neq j$ such that $\thsigma{1}{\A_i}=\thsigma{1}{\A_j}$. Towards a  contradiction, suppose that $\K$ is $\erange$-learnable, and let  $\Gamma$ be a continuous function reducing $\ldofk$ to $\erange$. Without loss of generality, we may assume that there exists $z \in \nats$ such that $z \in \range(\Gamma(\A_i))\setminus \range(\Gamma(\A_j))$. Consider any isomorphic copy $\esse$ of $\A_i$. Let $s$ be such that $z \in \Gamma(\esse\restriction_s)$: if there is no such $s$, then $\range(\Gamma(\esse)) \neq \range(\Gamma(\A_i))$ and hence, $\K$ is not $\erange$-learnable. If at stage $s$ we have that $z \in \Gamma(\esse\restriction_s)$, since by hypothesis for any $t$, $\thsigma{1}{\esse\restriction_t}\subseteq \thsigma{1}{\A_j}$, we can extend $\esse\restriction_s$ to a copy $\mathcal{T}$ of $\A_j$. Again, this shows that $\K$ is not $\erange$-learnable as  $z \in \range(\Gamma(\mathcal{T})) \neq \range(\Gamma(\A_j))$.

			For the opposite direction, suppose that $\K$ is a $\sigmainf{1}$-partial order, i.e., for every $i \neq j$ $\thsigma{1}{\A_i} \neq \thsigma{n}{\A_j}$. If $\thsigma{1}{\A_i} \not\subseteq\thsigma{1}{\A_j}$ then, by definition of $\sigmainf{1}$-partial order, there exists a $\sigmainf{1}$ formula $\varphi_{ij}$ such that $\A_i \models\varphi_{ij}$ and $\A_j \not\models \varphi_{ij}$.
			To prove that $\K$ is $\erange$-learnable, we need to show that $\ldofk$ is reducible to $\erange$ via some continuous $\Gamma$. Our $\Gamma$ will be a Turing operator relative to an oracle that encodes both the information of which $\varphi_{ij}$'s are defined and their definitions. This immediately implies that $\Gamma$ is also continuous (see \thref{lemma:folklore}).  For any $\esse \in \ldofk$ and for any $s,i,j \in \nats$ let
			\[\Gamma(\esse)(\str{s,i,j})= \begin{cases}
				\str{i,j} & \text{if } \varphi_{ij} \text{ is defined and }  \esse\restriction_s \models\varphi_{ij}\\
				\str{0,0} & \text{otherwise.}
			\end{cases}\]
			Notice that for any $\esse$, $\str{0,0} \in \Gamma(\esse)$: for example, $\Gamma(\esse)(s,i,i)=\str{0,0}$ for every $s,i \in \nats$.

			We now prove that $\Gamma$ is the desired reduction. Fix $\A_i \in \K$ and let $\esse \in \ldofk$:
			\begin{itemize}
				\item if $\esse \cong\A_i$, then just applying the definition of $\Gamma$ immediately implies that $\range(\Gamma(\esse)) = \range(\Gamma(\A_i))$ and hence that $\Gamma(\esse) \ \erange \ \Gamma(\A_i)$;
				\item  if $\esse \cong \A_k$ for $k \neq i$, there are two cases: either $\thsigma{1}{\A_k} \subseteq \thsigma{1}{\A_i}$ or $\thsigma{1}{\A_k} \not\subseteq \thsigma{1}{\A_i}$. In the first case since $\K$ is a $\sigmainf{1}$-partial order  we have that $\thsigma{1}{\A_i}\not\subseteq\thsigma{n}{\A_k}$. This implies that $\A_i \models\varphi_{ik}$ and $\A_k \not\models\varphi_{ik}$ and hence, since $\esse \cong \A_k$ we have that $\str{i,k} \in \range(\Gamma(\A_i))\setminus \range(\Gamma(\esse))$. In the second case, we have that $\A_k \models\varphi_{ki}$ and $\A_i \not\models\varphi_{ki}$: hence, since $\esse\cong\A_k$ we have that $\str{k,i} \in \range(\Gamma(\esse))\setminus \range(\Gamma(\A_i))$. In both cases  $\Gamma(\esse) \ \cancel{\erange} \ \Gamma(\A_i)$. 
			\end{itemize}
			This concludes the proof.
		\end{proof}

		The fact that $\thsigma{1}{\omega}=\thsigma{1}{\omega^*}$ together with \thref{theorem:characterization_Erange} leads to the following corollary.
		
		\begin{corollary}
			\thlabel{corollary:omega_omega*_E_range}
			$\{\omega,\omega^*\}$ is not $\erange$-learnable.
		\end{corollary}
		
		\begin{proposition}
			\thlabel{proposition:Erange_finite_families}
			$\erange \strictlearnreducible{\finitary} E_0$.
		\end{proposition}
		\begin{proof}
			Let $\K\defas \{\A_i: i < n\}$ with $n>0$ be an $\erange$-learnable family and notice that, by \thref{theorem:characterization_Erange}, $\K$ is a $\sigmainf{1}$-partial order. Without loss of generality, we may assume the following: if $i \neq j$ and $\thsigma{1}{\A_i} \subseteq\thsigma{1}{\A_j}$, then $i < j$. Given $\esse \in \ldofk$, we define $\learnerM$ as follows. At stage $s$
			\[\learnerM(\esse\restriction_s) \defas \begin{cases}
				\code{\A_i} & \text{if } i = \min\{ j < n: \esse\restriction_s \hookrightarrow \A_j\} \\
				? & \text{otherwise.}
			\end{cases}\]
			To show that $\learnerM$ $\ex$-learns $\K$, suppose that $\esse \in \ldofk$ is such that $\esse\cong \A_i$ for some $i <n$. Let $s_0\defas\min\{ s : (\forall j <  i)(\esse\restriction_s \not\hookrightarrow \A_j)\}$: such an $s_0$ exists as $\A_i$ is not embeddable in any $\A_j$ for $j<i$. From stage $s_0$ on, $\learnerM$ will not change its mind, since $\esse \cong \A_i$, and thus for every $s \geq s_0$, $\min\{ j < n: \esse\restriction_{s} \hookrightarrow \A_j\}= i$.
			
			For strictness, it suffices to combine the facts that $E_0$-learnability coincides with $\ex$-learning (\cite[Theorem 3.1]{bazhenov2021learning}) and $\{\omega,\omega^*\}$ is $\ex$-learnable (\thref{proposition:omega_vs_omega*}) but not $\erange$-learnable (Corollary~\thref{corollary:omega_omega*_E_range}).
		\end{proof}

		\begin{theorem}
			\thlabel{theorem: Erange_incomparable_E0}
			$\erange \strictlearnincomparable{} E_0$.
		\end{theorem}
		\begin{proof}
			The fact that $E_0 \not\learnreducible{} \erange$ follows from \thref{proposition:Erange_finite_families}.
			
			For the opposite direction, let $\K\defas \{\mathsf{R}_n \oplus \mathsf{I} : n \geq 2\} \cup \{\mathsf{R} \oplus \mathsf{I}\}$ and suppose that $\K$ is $\ex$-learnable by some learner $\learnerM$. We define $\esse \in \ldofk$ that $\learnerM$ fails to $\ex$-learn as follows. At every stage $s$, $\esse\restriction_s$ is such that $\esse\restriction_s\cong \mathsf{R}_{n_s} \oplus \mathsf{I}_{t_s}$ for some $n_s, t_s \leq s$ and we say that $2s\geq 2$ is an \emph{expansionary} stage if $\learnerM(\esse\restriction_{2s})=\code{\mathsf{R}_{n_{2s}}}$.
			The fact that we do not consider stages less than $2$ and that we just take care of even stages is just technical and needed for the definition of $\esse$: indeed, we want to ensure that $\esse$ is in $\ldofk$, and hence we use stages less than $2$ to add a copy of $\mathsf{R}_2$ in $\esse$ and odd stages to add pairwise incomparable elements, without caring of what $\learnerM$ does in these stages.
			We define $\esse$ as follows: Let $\esse\restriction_2 \cong \mathsf{R}_2 \oplus  \mathsf{I}_1$, and suppose that we have defined $\esse\restriction_{2s} \cong \mathsf{R}_{n_{2s}} \oplus \mathsf{I}_{t_{2s}}$ for $n_{2s},t_{2s} \leq 2s$ with $s \geq 1$. At stage $2s+1$ just let $\esse\restriction_{2s+1} \cong \mathsf{R}_{n_{2s}} \oplus \mathsf{I}_{{2s}+1}$. At stage $2s+2$:
			\begin{itemize}
				\item if $2s$ is an expansionary stage, let  $\esse\restriction_{2s+2} \cong \mathsf{R}_{n_{2s}+1} \oplus \mathsf{I}_{t_{2s}+1}$;
				\item if $2s$ is not an expansionary stage (i.e., $\learnerM(\esse\restriction_{2s}) \neq \code{\mathsf{R}_{n_{2s}}}$), let $\esse\restriction_{2s+2} \cong \mathsf{R}_{n_{2s}}\oplus \mathsf{I}_{t_{2s}+2}$;
			\end{itemize}
			Notice that:
			\begin{itemize}
				\item if there are infinitely many expansionary stages then $\esse\cong \mathsf{R} \oplus \mathsf{I}$;
				\item if there are only finitely many expansionary stages, say $n$, then $\esse \cong \mathsf{R}_{n+2} \oplus \mathsf{I}$. 
			\end{itemize}
			We claim that $\learnerM$ fails to $\ex$-learn $\esse$.
			\begin{itemize}
				\item Suppose that there are infinitely many expansionary stages (i.e., suppose that $\esse \cong \mathsf{R} \oplus \mathsf{I}$). Notice that for any two expansionary stages $2s' > 2s$, we have that $\esse\restriction_{2s} \cong \mathsf{R}_{n_{2s}} \oplus \mathsf{I}_{t_{2s}}$ and $\esse\restriction_{2s'} \cong \mathsf{R}_{n_{2s'}} \oplus \mathsf{I}_{t_{2s'}}$ for $n_{2s'}>n_{2s}$ and $t_{2s'} \geq t_{2s}$. We claim that $\learnerM(\esse\restriction_{2s'}) \neq \learnerM(\esse\restriction_{2s})$: this follows easily from the definition of expansionary stage, i.e., we have that $\learnerM(\esse\restriction_{2s})=\code{\mathsf{R}_{2s} \oplus \mathsf{I}}$ while   $\learnerM(\esse\restriction_{2s'})=\code{\mathsf{R}_{2s'}\oplus \mathsf{I}}$. Since by construction there are infinitely many expansionary stages, what we have just shown proves that $\learnerM$ change its mind infinitely many times, i.e., $\learnerM$ fails to $\ex$-learn $\esse$.
				\item Suppose that there are $n$ many expansionary stages (i.e., $\esse \cong \mathsf{R}_{n+2}\oplus \mathsf{I}$). Let $2s$ be the greatest expansionary stage: We claim that for all $2s' \geq 2s$, $\learnerM(\esse\restriction_{2s'}) \neq \code{\mathsf{R}_{n+2} \oplus \mathsf{I}}$  i.e., $\learnerM$ fails to $\ex$-learn $\esse$. To prove this, suppose that there is some stage $2s'>2s$ such that $\learnerM(\esse\restriction_{2s'}) = \code{\mathsf{R}_{n+2}\oplus \mathsf{I}}$: then, by definition, this would have been an expansionary stage, implying that $\mathsf{R}_{n+3} \hookrightarrow\esse\restriction_{2s'}$ and contradicting the fact that $\esse \cong \mathsf{R}_{n+2} \oplus \mathsf{I}$.
				
			\end{itemize}
			
			To conclude the proof, we need to show that $\K$ is $\erange$-learnable: it is easy to notice that $\K$ is a $\sigmainf{1}$-partial order and, by \thref{theorem:characterization_Erange}, this concludes the proof.
		\end{proof}
		
		To finish the picture around $\erange$ we give the following proposition.
		
		\begin{proposition}
			$\erange\strictlearnreducible{}{} E_3$.
		\end{proposition}
		\begin{proof}
			Let $\K=\{\A_i:i \in \nats\}$ be an $\erange$-learnable family. By \thref{theorem:characterization_Erange}, we have that $\K$ is a $\sigmainf{1}$-partial order.
			As in the proof of \thref{theorem:characterization_Erange}, let $\varphi_{ij}$ be a $\sigmainf{1}$-formula such that $\A_i \models\varphi_{ij}$ and $\A_j \not\models\varphi_{ij}$, if such a formula exists. 
			
			Our reduction $\Gamma$ from $\erange$ to $E_3$ will be a Turing operator relative to an oracle that encodes both the information of which $\varphi_{ij}$'s are defined and their definitions. This immediately implies that it is also continuous (see \thref{lemma:folklore}). Given $\esse \in \ldofk$, for every $i, j$ such that either $i=j$ or $i \neq j$ and $\thsigma{1}{\A_i}\subseteq\thsigma{1}{\A_j}$ let $\Gamma(\esse)^{[\str{i,j}]}(s)=0$ for every $s \in \nats$. For every $i \neq j$ such that $\thsigma{1}{\A_i}\not\subseteq \thsigma{1}{\A_j}$, at stage $s$, let
			\[\Gamma(\esse)^{[\str{i,j}]}(s) =\begin{cases}
				0 & \text{if } \esse\restriction_s \not\models\varphi_{ij}\\
				1 & \text{otherwise}
			\end{cases}\]
			Now the proof follows the schema of  \thref{theorem:characterization_Erange}.  Fix some $\A_i \in \K$: 
			\begin{itemize}
				\item if $\esse \cong \A_i$, then for every $j \neq k$ such that $\varphi_{kj}$ is defined and $\A_i \models\varphi_{kj}$, \\ $\Gamma(\esse)^{[\str{k,j}]} \ E_0 \ 1^\nats \ E_0 \ \Gamma(\A_i)^{[\str{k,j}]}$ while all the other columns are of the form $0^\nats$. Hence, for every $n$, $\Gamma(\esse)^{[n]} \ E_0 \ \Gamma(\A_i)^{[n]}$, i.e., $\Gamma(\esse) \ E_3 \ \Gamma(\A_i)$.
				
				\item  if $\esse \cong \A_k$ for some $k \neq i$ either $\thsigma{1}{\A_k}\subseteq\thsigma{1}{\A_i}$ or $\thsigma{1}{\A_k} \not\subseteq \thsigma{1}{\A_i}$. In the first case, since $\K$ is a $\sigmainf{1}$-partial order, we have that $\thsigma{1}{\A_i}\not\subseteq \thsigma{1}{\A_k}$. This implies that $\A_i \models \varphi_{ik}$ and $\A_k \not\models\varphi_{ik}$: since $\esse \cong \A_k$ we have that $\Gamma(\esse)^{[\str{i,k}]}=0^\nats$ while $\Gamma(\A_i)^{[\str{i,k}]} \ E_0\ 1^\nats$, i.e.\ $\Gamma(\esse) \ \centernot{E_3} \ \Gamma(\A_i)$. In the second case, we have that 
				$\A_k \models\varphi_{ki}$ and $\A_i \not\models\varphi_{ki} \A_i$: since $\esse \cong \A_k$, we have that  
				$\Gamma(\esse)^{[\str{k,i}]} \ E_0\ 1^\nats$ while $\Gamma(\A_i)^{[\str{k,i}]}=0^\nats$, i.e., $\Gamma(\esse) \ \centernot{E_3} \ \Gamma(\A_i)$.
			\end{itemize}
			This means that $\esse \cong \A_i$ if and only if $\Gamma(\esse) \ {E_3} \ \Gamma(\A_i)$, i.e., the family $\K$ is $E_3$-learnable. 
			
			The strictness of the reduction now follows from \thref{proposition:Erange_finite_families} and the fact that $E_3 \learnequiv{\finitary} E_0$ (\cite[Theorem 5.1]{bazhenov2021learning}). This concludes the proof.
		\end{proof}
		\subsection{non-U-shaped and decisive learning}
		
		We first give the definitions for $\nonushape$-learning and $\decisive$-learning.
		\begin{definition}
			\thlabel{definition:nonushapedec}
			Let $\K$ be a  family of structures.
			\begin{itemize}
				\item  A learner $\learnerM$ $\nonushape$-learns $\K$ if,  it $\ex$-learns $\K$ and for every $\esse\in\ldofk$, it never abandons the correct conjecture. That is, given $n_0\defas \min \{n : \learnerM(\esse\restriction_n)=\ulcorner \A \urcorner\}$ 
				\[\esse\cong \A \implies (\forall m>n_0)(\learnerM(\esse\restriction_m)=\ulcorner \A \urcorner). \]
				
				\item    A learner $\learnerM$ $\decisive$-learns $\K$ if, it $\ex$-learns $\K$ and for every $\esse\in\ldofk$, it never returns to a previously abandoned conjecture. That is, for every $\ulcorner \A \urcorner$ such that $\A \in \K$ and for every $n$
				\[\learnerM(\esse\restriction_n)=\ulcorner \A \urcorner \land \learnerM(\esse\restriction_{n+1}) \neq \ulcorner \A \urcorner \implies (\forall m>n)(\learnerM(\esse\restriction_{m}) \neq \ulcorner \A \urcorner).\]
			\end{itemize}
		\end{definition}
		
		The next proposition shows that, in our framework, the two paradigms introduced above actually coincide.
		
		\begin{proposition}
			\thlabel{proposition:nonushapedecequiv}
			A family of structures $\K$ is $\nonushape$-learnable if and only if it is $\decisive$-learnable.
		\end{proposition}
		\begin{proof}
			The right-to-left direction is trivial. For the opposite direction, suppose that $\K$ is $\nonushape$-learnable by some learner $\learnerM_U$. We define a learner $\learnerM_D$ that $\decisive$-learns $\K$ as follows. Let $\esse \in \ldofk$ and at stage $0$, let $\learnerM_D(\esse\restriction_0)=\learnerM_U(\esse\restriction_0)$. At stage $s+1$, let $\learnerM_D(\esse\restriction_{s+1}) \defas$
			\[ 
			\begin{cases}
				\learnerM_U(\esse\restriction_{s+1}) & \text{if } \learnerM_U(\esse\restriction_{s+1})=\learnerM_D(\esse\restriction_{s}) \text{ or}\\
				& \learnerM_U(\esse\restriction_{s+1}) \neq \learnerM_U(\esse\restriction_{s}) \land (\forall i<s+1)(\learnerM_U(\esse\restriction_i) \neq \learnerM_U(\esse\restriction_{s+1}))\\
				\learnerM_D(\esse\restriction_{s})  & \text{if } \learnerM_U(\esse\restriction_{s+1}) \neq \learnerM_D(\esse\restriction_{s}) \land (\exists i<s+1)(\learnerM_U(\esse\restriction_i) =\learnerM_U(\esse\restriction_{s+1})).\\
			\end{cases}\]

			The definition of $\learnerM_D$ ensures that an abandoned hypothesis is never outputted again, and the fact that $\learnerM_U$ $\nonushape$-learns $\K$ ensures that $\learnerM_D$ eventually stabilizes to the correct conjecture.
		\end{proof}
		
		We now give the promised characterization of $\nonushape$-learnability (and hence, by \thref{proposition:nonushapedecequiv} of $\decisive$-learnability). We first need the following definition.
		
		\begin{definition}
			\thlabel{definition:solidposets}
			Let $\K$ be a  $\sigmainf{1}$-partial order. We say that $\K$ is a \emph{solid $\sigmainf{1}$-partial order} if, for any nonempty $\A \in \K$
			\[\thsigma{1}{\A} \setminus \bigcup \{\thsigma{1}{\B}: \B \in \K \land \thsigma{1}{\B}  \subset \thsigma{1}{\A} \} \neq \emptyset.\]
			
		\end{definition}
		Informally, in a solid  $\sigmainf{1}$-partial order each structure $\A$ has formula that separates $\A$ from its proper lower cone (with respect to the $\sigmainf{1}$-theories).
		\begin{theorem}
			\thlabel{theorem:solidposets}
			A family of structures $\K$ is $\nonushape$-learnable if and only if $\K$ is a solid $\sigmainf{1}$-partial order.
		\end{theorem}
		\begin{proof}
			For the left-to-right direction, suppose that $\K$ is $\nonushape$-learnable by some learner $\learnerM$ and assume that $\K$ is not a solid $\sigmainf{1}$-partial order. In case $\K$ is not a $\sigmainf{1}$-partial order then we can easily show that $\K$ is not $\nonushape$-learnable. So assume that $\K$ is a $\sigmainf{1}$-partial order but not a solid one: let $\A \in \K$ be such that $\thsigma{1}{\A} \setminus\bigcup \{\thsigma{1}{\B}: \B \in \K \text{ and } \thsigma{1}{\B}  \subset \thsigma{1}{\A}\} = \emptyset$. We define $\esse \in \ldofk$ letting $\esse\restriction_s \cong \A\restriction_s$ where $s= \min\{t: \learnerM(\esse\restriction_t)=\code{\A}\}$ (notice that such an $s$ must exist otherwise $\learnerM$ fails to even $\ex$ learn $\K$). By assumption there is some $\B \in \K$ such that $\B \not \cong \A$ and $ \thsigma{1}{\B}  \subset \thsigma{1}{\A}$ and $\esse\restriction_s$ can be extended to $\B$. Then start extending $\esse\restriction_s$ to a copy of $\B$. Since $\learnerM$ $\nonushape$-learns $\K$, there is a stage $s'>s$ such that $\learnerM(\esse\restriction_{s'})=\code{\B}$. Since $\thsigma{1}{\B}\subset \thsigma{1}{\A}$ and since $\esse\restriction_{s'}$ is finite we can extend again $\esse$ to a copy of $\A$, forcing $\learnerM$ to get back to output $\code{\A}$ and witnessing the desired contradiction.

			For the right-to-left direction, suppose that $\K=\{\A_i:i \in \nats\}$ is a solid $\sigmainf{1}$-partial order. For any $\A_i \in \K$, let $\varphi_i$ be the $\sigmainf{1}$ formula such that $\A_i \models \varphi_i$ and for any $j$ such that $\thsigma{1}{\A_j}\subset \thsigma{1}{\A_i}$, $\A_j \not\models \varphi_i$.  
			
			Given $\esse \in \ldofk$ and a stage $s \in \nats$,
			let \[C_s:=\{\code{\A_i } : \thsigma{1}{\esse\restriction_s} \subseteq \thsigma{1}{\A_i} \land \esse\restriction_s \models \varphi_i\}.\] 
			We define a learner $\learnerM$ as follows. At stage $0$ let $\learnerM(\esse\restriction_0)=?$ and at stage  $s+1$, let $\learnerM(\esse\restriction_{s+1}) \defas$
			\[
			\begin{cases}
				\learnerM(\esse\restriction_{s}) & \text{if } C_s=\emptyset \lor\big(\learnerM(\esse\restriction_{s}) =\code{\A_i} \text{ for some }i \text{ and } \\
				& \thsigma{1}{\esse\restriction_{s+1}}\subseteq \thsigma{1}{\A_i}\big)\\
				\min\{\code{\A_i} : \code{\A_i} \in C_s\} & \text{otherwise.}
			\end{cases}
			\]
			Notice that, by definition, $\learnerM$ changes its mind from a conjecture different from $?$ at stage $s$ if and only if $\thsigma{1}{\esse\restriction_s} \not\subseteq\thsigma{1}{\A_i}$ where $\code{\A_i}=\learnerM(\esse\restriction_{s-1})$: this implies that if $\esse \cong \A_i$ and at some stage $s$ $\learnerM(\esse\restriction_s)=\code{\A_i}$ then for all $t \geq s$, $\learnerM(\esse\restriction_t)=\code{\A_i}$. Hence, to conclude the proof, it remains to show that for any $\esse \in \ldofk$ and for any $i \in \nats$, if $\esse \cong \A_i$, there exists some stage $t$ such that $\learnerM(\esse\restriction_t)=\code{\A_i}$.  Suppose that $\esse \cong \A_i$ and notice that there exists some stage $s$ such that $\code{\A_i} \in C_{s}$. If $\learnerM(\esse\restriction_{s})=\code{\A_i}$ then there is nothing to prove. Otherwise, assume $\learnerM(\esse\restriction_{s})=\code{\A_j}$ for some $j \neq i$ and suppose that for every $s' \geq s$,  $\learnerM(\esse\restriction_{s'})=\code{\A_j}$. By $\learnerM$'s definition, this means for any $s' \geq s$,  $\thsigma{1}{\esse\restriction_{s'}} \subseteq \thsigma{1}{\A_j}$, and since $\esse \cong \A_i$ we have that $\thsigma{1}{\A_i}\subseteq\thsigma{1}{\A_j}$. On the other hand, $\learnerM(\esse\restriction_s)=\code{\A_j}$ implies that $\code{\A_j} \in C_s$ and hence $\esse \models \varphi_j$ which cannot be the case as $\esse \cong \A_i$ and $\A_i \not \models \varphi_j$ for every $j$ such that $\thsigma{1}{\A_i} \subseteq \thsigma{1}{\A_j}$. Repeating the same argument for every structure $\code{\A_k}$ such that $\code{\A_k}<\code{\A_i}$ and $\learnerM(\esse\restriction_{s'})=\code{\A_k}$ for some $s' \geq s$, we get that there exists some stage $s$ such that $\learnerM(\esse\restriction_s)=\code{\A_i}$. This together with the fact that $\learnerM$ changes its mind from a conjecture different from $?$ at stage $s$ if and only if $\thsigma{1}{\esse\restriction_s} \not\subseteq\thsigma{1}{\A_i}$ concludes the proof.
		\end{proof}
		
		The following proposition summarizes the relationships between the learning paradigms considered so far.
		\begin{proposition}
			\thlabel{proposition:nonushapeerange}
			The following relations between $\col$, $\nonushape$, $\ex$ and $\erange$ hold:
			\begin{enumerate}[(i)]
				\item $\col\strictlearnreducible{} \nonushape$ and $\col\strictlearnreducible{\mathrm{fin}} \nonushape$;
				\item $\nonushape \strictlearnreducible{} \erange$ and $\nonushape \learnequiv{\finitary} \erange$;
				\item  $\nonushape \strictlearnreducible{\finitary} \ex$ and hence, in particular $\nonushape \strictlearnreducible{} \ex$.
				
			\end{enumerate}
		\end{proposition}
		\begin{proof}
			To prove $\mathbf{(i)}$, we first show that $\col \learnreducible{} \nonushape$. Suppose that $\K=\{\A_i:i \in \nats\}$ is $\col$-learnable by a learner $\learnerM$. We define a learner $\learnerM'$ that $\nonushape$-learns $\K$ letting $\learnerM'(\esse\restriction_s):=\min\{\code{\A_i}: (\forall t \leq s)(\learnerM(\esse\restriction_t) \neq \code{\A_i})\}$. To see that $\learnerM'$ $\nonushape$-learns $\K$ suppose that $\esse \cong \A_i$. Let $s=\min\{t : (\forall \code{\A_j} < \code{\A_i})(\exists t_{j}<t)(\learnerM(\esse\restriction_{t_j})=\code{\A_j})\}$: notice that such an $s$ exists since, by hypothesis, $\learnerM$ $\col$-learns $\K$. Then for all $s' \geq s$ $\learnerM'(\esse\restriction_{s'})=\code{\A_i}$. By $\learnerM'$ definition it is also easy to check that for any $s''< s$, $\learnerM'(\esse\restriction_{s''}) \neq \code{\A_i}$.
			
			To see that $\nonushape \not\learnreducible{\mathrm{fin}} \col$ (and in particular, $\nonushape \not\learnreducible{} \col$), it suffices to notice that the family $\{\tilde{L_3}, \tilde{L_4}\}$ is $\nonushape$-learnable but not $\col$-learnable. Indeed, to prove $\nonushape$-learnability, given $\esse \in \ldofk$ we define a learner that always outputs $\tilde{L}_3$ except if, at some stage $s$, $\esse\restriction_s$ contains a linear order of four elements: clearly such a learner $\nonushape$-learns $\K$.  To show that $\K$ is not $\col$-learnable it suffices to notice that $\thsigma{1}{\tilde{L}_3}\subseteq\thsigma{1}{\tilde{L}_4}$ and apply \thref{theorem:12antichains}.

			To prove $\mathbf{(ii)}$, $\nonushape\learnreducible{}\erange$ follows from \thref{theorem:characterization_Erange} and \thref{theorem:solidposets} which says that $\nonushape$ and $\erange$ respectively learn all and only solid $\sigmainf{1}$-partial orders and $\sigmainf{1}$-partial orders:
			a solid $\sigmainf{1}$-partial order is in particular a $\sigmainf{1}$-partial order and hence we are done. To prove that the reduction is strict it suffices to notice that not every $\sigmainf{1}$-partial order is solid: i.e., take $\{\tilde{L}_n: n \in \nats\} \cup \{\tilde{\omega}\}$ and notice that $\thsigma{1}{\tilde{\omega}}\setminus \bigcup\{\thsigma{1}{\B}: \B \in \K \land \thsigma{1}{\B}\subset \thsigma{1}{\tilde{\omega}}\}=\emptyset$. 
			
			To show that $\nonushape \learnequiv{\finitary}\erange$ it remains to show that $\erange \learnreducible{\finitary} \nonushape$. To do so, notice that the learner $\learnerM$ defined in \thref{proposition:Erange_finite_families} to show that $\erange\learnreducible{\mathrm{fin}} \ex$ also witnesses that $\erange\learnreducible{\mathrm{fin}} \nonushape$. 
			
			We now prove $\mathbf{(iii)}$. The fact that $\nonushape$-learnability implies $\ex$-learnability is trivial as $\nonushape$-learnability is a restriction of $\ex$-learnability. 
			
			To see that the reduction is strict also when restricted to finite families just notice that $\{\omega,\omega^*\}$ is not even $\erange$-learnable (\thref{corollary:omega_omega*_E_range}). 
		\end{proof}

		We now give further results on $\nonushape$ learning.

		\begin{proposition}
			For any fixed $\alpha$, there exists a family that is $\nonushape$-learnable but not $\alpha$-learnable.
		\end{proposition}
		\begin{proof}
			Using \cite[Theorem 3]{bazcipsan_cie}, it is easy to define an $\Id$-learnable family that is not $\alpha$-learnable. Combining \thref{corollary:finidco} and \thref{proposition:nonushapeerange} we know that  $\Id \strictlearnreducible{} \nonushape$ and hence we obtain that $\K$ is $\nonushape$-learnable. 
		\end{proof}

		Before concluding this section we observe the following. In the previous subsection, we have shown that $E_0$-learnability and $\erange$-learnability give us the first example of an incomparability between two criteria in our learning hierarchy. It is natural to ask whether these two notions have a meet in the learning hierarchy and notice that $\nonushape$-learning is a natural candidate. Indeed, the fact that $\erange$-learnability coincides with asking that $\K$ is a $\sigmainf{1}$-partial order (\thref{theorem:characterization_Erange}) seems to suggest that a learner can adjust its conjecture depending on how the given structure is extended, without needing to get back to previous conjectures. In addition, $E_0$-learnability suggests that this process should eventually stop. The following family is a counterexample to this intuition as shown in \thref{proposition:not_the_meet}. We define the family of partial orders $\mathfrak{P}=\{\tilde{\mathcal{P}_i} : i \in \nats\}$, where
		for $k>0$,
		\begin{itemize}
			\item $\mathcal{P}_k$ is the partial order $(\{0,\dots,2k+1\},\leq_{\mathcal{P}_k})$, where for $i<k$, $2i\leq_{\mathcal{P}_k} 2i+2$ and for $i \leq k$ $2i \leq_{\mathcal{P}_k} 2i+1$, and
			\item $\mathcal{P}_0$ is the partial order $(\nats,\leq_{\mathcal{P}_0})$, where for every $i$, $2i\leq_{\mathcal{P}_0} 2i+2$ and for every $j>0$,  $2j \leq_{\mathcal{P}_0} 2j-1$.
		\end{itemize}
		Notice that $\thsigma{1}{\tilde{\mathcal{P}_1}}\subseteq \thsigma{1}{\tilde{\mathcal{P}_2}}\subseteq \dots \subseteq \thsigma{1}{\mathcal{P}_0}$.
		\begin{table}[H]
			\centering
			\begin{tabular}{c|c}
				\begin{tikzpicture}[scale=1, main/.style = {draw, circle}] 				
					\node[main] (0) {0}; 
					\node[main] (2) [right of=0,xshift=0.5cm] {2}; 
					\node[main] (4) [right of=2,xshift=0.5cm] {4}; 
					\node[main] (1) [below of=0,yshift=-0.5cm] {1}; 
					\node[main] (3) [below of=2,yshift=-0.5cm] {3}; 
					\node[main] (5) [below of=4,yshift=-0.5cm] {5}; 
					\draw[->] (0) -- (2);
					\draw[->] (2) -- (4);
					\draw[->] (0) -- (1);
					\draw[->] (2) -- (3);
					\draw[->] (4) -- (5);
				\end{tikzpicture}&
				\begin{tikzpicture}[scale=2.0, main/.style = {draw, circle}] 
					\node[main] (0) {0}; 
					\node[main] (2) [right of=0,xshift=0.5cm] {2}; 
					\node[main] (4) [right of=2,xshift=0.5cm] {4}; 
					\node[main] (6) [right of=4,xshift=0.5cm] {6}; 
					\node[] (dots) [right of=6,xshift=0.5cm] {$\dots$}; 
					\node[main] (1) [below of=2,yshift=-0.5cm] {1}; 
					\node[main] (3) [below of=4,yshift=-0.5cm] {3}; 
					\node[main] (5) [below of=6,yshift=-0.5cm] {5}; 
					\node[] (dots2) [right of=5,xshift=0.5cm] {$\dots$}; 
					\draw[->] (0) -- (2);
					\draw[->] (2) -- (4);
					\draw[->] (4) -- (6);
					\draw[->] (6) -- (dots);
					\draw[->] (2) -- (1);
					\draw[->] (4) -- (3);
					\draw[->] (6) -- (5);
				\end{tikzpicture} \\
			\end{tabular}
			\caption{On the left-hand side a graphical representation of $\mathcal{P}_2$ and on the right-hand side a graphical representation of $\mathcal{P}_0$.}
		\end{table}
		\begin{proposition}
			\thlabel{proposition:not_the_meet}
			$\mathfrak{P}$ is both $E_0$-learnable and $\erange$-learnable, but it is not $\nonushape$-learnable.
		\end{proposition}
		\begin{proof}
			We first show that  $\mathfrak{P}$ is $E_0$-learnable. By \cite[Theorem 3.1]{bazhenov2021learning} this is the same as showing that $\mathfrak{P}$ is $\ex$-learnable. We define an $\ex$ learner $\learnerM$ as follows. Given $\esse \in \ldofk$, at a stage $s$, we have that  $\learnerM(\esse\restriction_s) = \code{\tilde{\mathcal{P}_0}}$ if
			\[(\exists a,b \in \esse\restriction_s)(\forall c \in \esse\restriction_s)[a \leq_{\esse} b \land b \nleq_{\esse} a \land (c = a \lor   b \leq_{\esse} c  \lor (\forall d \in \esse\restriction_s)[ c=d \lor c \mid_{\esse} d])].\]
			Otherwise, $\learnerM(\esse\restriction_s)\defas \code{\tilde{\mathcal{P}_n}}$ where $n \defas \min\{i : \esse\restriction_s \hookrightarrow \mathcal{P}_i \text{ and }i>0 \}$.
			
			Informally, $\learnerM$ outputs $\code{\tilde{\mathcal{P}_0}}$ if $\esse\restriction_s$ contains an element $a$ that \lq\lq behaves\rq\rq\ like the element $0$ in $\mathcal{P}_0$ and an element $b$ that \lq\lq behaves\rq\rq\ like the element $2$ in $\mathcal{P}_0$. If such an element does not exist at stage $s$, $\learnerM$ outputs the \lq\lq smallest\rq\rq\ $\mathcal{P}_n$ with $n>0$ on which $\esse\restriction_s$ embeds into. 
			
			To show that $\mathfrak{P}$ is $\erange$-learnable it is easy to check that, for any $i \neq j$ $\thsigma{1}{\tilde{\mathcal{P}_i}} \neq \thsigma{1}{\tilde{\mathcal{P}_j}}$, i.e., $\mathfrak{P}$ is a $\sigmainf{1}$-partial order. By \thref{theorem:characterization_Erange}, we obtain that $\mathfrak{P}$ is $\erange$-learnable.
			
			We now show that $\mathfrak{P}$ is not $\nonushape$-learnable. To do so, suppose that there exists a learner $\learnerM$ which $\nonushape$-learns $\mathfrak{P}$. We construct a copy $\esse$ of $\tilde{\mathcal{P}_0}$ as follows. Start constructing $\esse$ as a copy of $\tilde{\mathcal{P}_0}$: if for all $s$, $\learnerM(\esse\restriction_s) \neq \code{\tilde{\mathcal{P}_0}}$, then $\learnerM$ fails even to $\ex$-learn $\mathfrak{P}$. Hence, there exists a stage $s_0$ such that $\learnerM(\esse\restriction_{s_0}) = \code{\tilde{\mathcal{P}_0}}$. After stage $s_0$, pick any $k$ such that the number of elements in $\esse\restriction_{s_0}$ is less than $k$, and start extending $\esse_{s_0}$ to $\esse'\cong\tilde{\mathcal{P}_k}$. If for all $s > s_0$, $\learnerM(\esse'\restriction_s) \neq \code{\tilde{\mathcal{P}_k}}$, then $\learnerM$ fails even to $\ex$-learn $\mathfrak{P}$. Hence, there exists a stage $s_1$ such that $\learnerM(\esse'\restriction_{s_1}) = \code{\tilde{\mathcal{P}_k}}$: at this point, since for any $k>0$, $\thsigma{1}{\tilde{\mathcal{P}_k}}\subseteq \thsigma{1}{\tilde{\mathcal{P}_0}}$, we can continue to extend $\esse'_{s_1}$ as a copy of $\tilde{\mathcal{P}_0}$, forcing $\learnerM$ to change its mind back to $\code{\tilde{\mathcal{P}_0}}$ proving that $\mathfrak{P}$ is not $\nonushape$-learnable.
		\end{proof}

		\section{Characterizing learnabilities in terms of $\sigmainf{2}$-formulas}
		\label{sec:sigma2}
		We now define another learning criterion for our paradigm that comes from classical algorithmic learning theory, namely \emph{partial learnability}, denoted by $\pl$.
		
		\begin{definition}
			\thlabel{definition:pl}
			A   family of structures $\K$ is $\pl$-\emph{learnable} if, for every $\esse\in\ldofk$, the learner outputs infinitely often a conjecture if and only if it is the correct one. That is, the set $\{n : \learnerM(\esse\restriction_n) = \code{\A} \}$ is infinite if and only if $\A \cong \esse$.
		\end{definition}
		Without loss of generality, we assume that a $\pl$-learner can output infinitely often  \lq\lq ?\rq\rq.
		
		\begin{theorem}
			\thlabel{theorem:plcharacterization}
			A family $\K$ is $\pl$-learnable if and only if $\K$ is a $\sigmainf{2}$-antichain.
		\end{theorem}
		
		\begin{proof}
			We first show the right-to-left direction. Suppose that $\K=\{\A_i : i\in\nats \}$ is a $\sigmainf{2}$-antichain. By \thref{lemma:easyposet}$(c)$, for any $i < j$, the family $\{\A_i, \A_j\}$ is a $\sigmainf{2}$-strong antichain and hence, by \cite[Theorem 3.1]{bazhenov2020learning}we obtain that the family $\{\A_i,\A_j\}$ is $\ex$-learnable and we denote the corresponding learner with $\learnerM_{i,j}$. We define the following counter for a learner $\learnerM$ in general. For $i,s \in \nats$, let 
			\[c(\learnerM,i,s) \defas  |\{t \leq s : \learnerM(\esse\restriction_t)=\code{\A_i}\}|.\]
			It is easy to notice that for any $\learnerM$, $i,s$, $c(\learnerM,i,s) \leq s+1$.

			A $\pl$-learner $\mathbf{L}$ for $\K$ is defined as follows: For any $\esse \in \ldofk$, let $\mathbf{L}(\esse\restriction_{\str{i,k}}) \defas$ \[
			\begin{cases}
				? & \text{if } k=0\\
				\code{\A_i} & \text{if } (\exists k' < k )(\forall j \leq k')
				( c(\mathbf{L},i,\str{i,k-1}) < k' \land
				c(\mathbf{L},i,\str{i,k-1}) < c(\learnerM_{i,{j}},i,k)
				)
				\\
				? & \text{otherwise.}
			\end{cases}
			\]
			Informally, in stages $\str{i,\cdot}$ $\mathbf{L}$ is taking care of $\A_i$. At stage $\str{i,k}$, $\mathbf{L}$ outputs  $\code{\A_i}$ if there is $k'<k$ such that $\learnerM_{i,0},\dots, \learnerM_{i,k'}$ all output $\code{\A_i}$ more than $c(\mathbf{L},i,\str{i,k-1})$ times and $\mathbf{L}$ has output $\code{\A_i}$ only  less than $k'$ times.

			Now, suppose that $\esse \cong \A_i$. Then, for all $j\neq i$, $\learnerM_{i,j}(\esse)$ will eventually stabilize to $\code{\A_i}$. We claim that this implies that for any $n$ there exists a $k\geq n+1$ such that $\mathbf{L}(\esse\restriction_{\str{i,k}})= \code{\A_i}$. To prove the claim, fix $n$ and let $k \defas 1 + \min \{ k' : (\forall j \leq n+1)(c(\learnerM_{i,j},i ,k' ) \geq n+1\})$ (notice that $k \geq n+1$). Such a $k$ must exist as all $\learnerM_{i,j}(\esse)$ will eventually stabilize to $\code{\A_i}$. By $\mathbf{L}$'s definition,  the second condition applies and $\mathbf{L}$ at stage $\str{i,k}$ will output $\code{\A_i}$.
			
			Suppose that $\esse \cong \A_j$ with $i\neq j$. Then there must be a stage $k$ such that for every $k' \geq k$ $\learnerM_{ij}(\esse \restriction_{k'})= \code{\A_j} \neq \code{\A_i}$. At such a stage $\learnerM_{ij}$ has outputted the conjecture $\code{\A_i}$ at most $k$ times and it will not output $\code{\A_i}$ anymore. Hence, by the second condition of $\mathbf{L}$'s definition,   
			it is clear that $\mathbf{L{}}$ cannot output $\code{\A_i}$ for more than $\max\{j+1,k+1\}$ times. This concludes the first part of the proof.

			For the left-to-right direction, suppose that there are $\A_i,\A_j\in \K$ so that $\thsigma{2}{\A_i}$ and $\thsigma{2}{\A_j}$ do not form an antichain and suppose by contradiction that $\K$ is $\pl$-learnable by some $\learnerM$.
			We show that from such a learner  one can define an $\ex$-learner $\learnerM^*$ learning $\{\A_i,\A_j\}$, contradicting \cite[Theorem 3.1]{bazhenov2020learning}. 
			
			To do so, for any $\esse \in \mathrm{LD}(\{\A_i,\A_j\})$ at stage $s$, let
			\[
			\learnerM^*(\esse\restriction_s)=\begin{cases}
				\learnerM(\esse\restriction_s) &\text{if $\learnerM(\mathcal{S})\in \{\code{\A_i},\code{\A_j}\}$};\\
				\learnerM(\esse\restriction_t) & \text{otherwise, where } t= \max\{t' \leq s : \learnerM(\esse\restriction_{t'})\in \{\code{\A_i},\code{\A_j}\}\}.
			\end{cases}
			\]
			It is clear that $\learnerM^*$ learns $\{\A_i,\A_j\}$ and this concludes the proof.
		\end{proof}
		
		\subsection{$\pl$-learnability under learning reducibility}
		
		\begin{proposition}
			\thlabel{proposition:E3pl}
			$E_3\learnreducible{} \pl$.
		\end{proposition}
		\begin{proof}
			Suppose that $\K=\{\A_i:i \in \nats\}$ is $E_3$-learnable, and let $\Gamma$ be a continuous reduction from $\ldofk$ to $E_3$. Clearly, for any $i \neq j$, we have that $\Gamma(\A_i) \centernot{E_3} \Gamma(\A_j)$:  we define the \emph{column of disagreement between $\A_i$ and $\A_j$} as $d_{ij}\defas \min \{ t : \Gamma(\A_i)^{[t]} \centernot{E_0} \Gamma(\A_j)^{[t]} \}$ and the \emph{sequence of disagreement between $\A_i$ and $\A_j$} as $p_{ij} \in \Baire$ letting $p_{ij}(0)=0$ and
			\[p_{ij}(s+1) \defas \min \{ k> p_{ij}(s) : \Gamma(\A_i)^{[d_{ij}]}(k) \neq \Gamma(\A_j)^{[d_{ij}]}(k) \}.\]
			Clearly, for every $i \neq j$, $d_{ij}=d_{ji}$ and $p_{ij}=p_{ji}$.
			Informally, $d_{ij}$ is the index of the first column in $\Gamma(\A_i)$ that is not $E_0$-equivalent to the corresponding column of  $\Gamma(\A_j)$, 
			and $p_{ij}$ traces the elements in which the $d_{ij}$-th column of $\Gamma(\A_i)$ and the $d_{ij}$-th column of $\Gamma(\A_j)$
			disagree (since $\Gamma(\A_i)^{[d_{ij}]} \centernot{E_0} \Gamma(\A_j)^{[d_{ij}]}$, there are infinitely many such elements).
			
			We define a learner $\learnerM$ which $\pl$-learns $\K$ as follows. First, $\learnerM$ for every $i$ outputs the conjecture $\code{\A_i}$ at least one time. Then for $n>1$, given $\esse \in \ldofk$, for every $i$, $\learnerM$ outputs the conjecture $\code{\A_i}$ for the $n$-th time if and only if for every $m < n$ such that $m\neq i$,
			\[
			(\exists k_m) (\forall t < n) (\Gamma(\esse)^{[d_{im}]}(p_{im}(k_m+t)) = \Gamma(\A_i)^{[d_{im}]}(p_{im}(k_m+t))).
			\]
			Informally, $\learnerM$ in order to decide to output $\code{\A_i}$ for the $n$-th time checks whether $\A_i$ is \lq\lq similar enough\rq\rq\ with $\esse$ compared to the first $n$-many structures. To do so, it checks if $\Gamma(\esse)$ agrees with $\Gamma(\A_i)$ on the first $n$-many columns of disagreement for $n$-many consecutive elements indexed by the corresponding sequence of disagreement. 
			
			The following easy observation is important to conclude the proof. If $\esse \cong \A_i$, then $\Gamma(\esse) \ E_3 \ \Gamma(\A_i)$, and hence, in particular, the $d_{ij}$-th columns, $j \in \nats$, of $\Gamma(\esse)$ and $\Gamma(\A_i)$ are all $E_0$-equivalent as well. 
			
			Suppose that $\esse \cong \A_i$: we need to show that $\code{\A_i}$ is the only conjecture outputted by $\learnerM$ infinitely often. To do so it suffices to show that:
			\begin{enumerate}[(i)]
				\item for every $n$, $\learnerM$ outputs at least $n$ many times the conjecture $\code{\A_i}$ and
				\item for every $j \neq i$, $\learnerM$ outputs only finitely many time the conjecture $\code{\A_j}$.
			\end{enumerate}
			To prove \textbf{(i)}, fix $n$ and notice that for every $m < n$ by inspecting larger and larger initial segments of $\Gamma(\esse)^{[d_{im}]}$, there will be a stage such that $\Gamma(\esse)$ and $\Gamma(\A_i)$ agree on $n$-many consecutive elements indexed by the sequence of disagreement $p_{im}$ (indeed, we have already observed that by the definition of $E_0$-equivalence, every column of $\Gamma(\esse)$ is eventually equivalent to the corresponding one in $\Gamma(\A_i)$): at such a stage $\learnerM$ outputs $\code{\A_i}$ for the $n$-th time. 
			
			To prove \textbf{(ii)}, consider the column of disagreement $d_{ij}$ and notice that \[\Gamma(\esse)^{[d_{ij}]} E_0 \Gamma(\A_i)^{[d_{ij}]} \centernot{E_0} \Gamma(\A_j)^{[d_{ij}]}.\] Hence there is some $k$ such that for all $k'>k$ 
			\[\Gamma(\esse)^{[d_{ij}]}(p_{ij}(k')) = \Gamma(\A_i)^{[d_{ij}]}(p_{ij}(k')) \neq \Gamma(\A_j)^{[d_{ij}]}(p_{ij}(k')),\]
			i.e., $\Gamma(\esse)^{[d_{ij}]}$ disagrees with all elements greater than $k$ indexed by the sequence of disagreement $p_{ij}$. This means that $\learnerM$ outputs the conjecture $\code{\A_j}$ at most $\max\{k+1,i+1\}$ times, and this concludes the proof.
		\end{proof}
		
		Combining \thref{proposition:E3pl} and \thref{theorem:prevwork} we obtain the following corollary.
		\begin{corollary}
			\thlabel{corollary:e0pl1}
			$\ex \learnequiv{\finitary} \pl$ and
			$\ex \strictlearnreducible{} \pl$
		\end{corollary}
		\begin{proof}
			The first equivalence follows from the fact $\sigmainf{2}$-strong antichains and -antichains coincide for finite classes (\thref{lemma:easyposet}$(c)$). The fact that $\ex \strictlearnreducible{} \pl$ follows from \thref{proposition:E3pl}.
		\end{proof}
		
		Since $\{\omega,\zeta\}$ is clearly not a $\sigmainf{2}$-antichain, by \thref{theorem:plcharacterization} we obtain that:
		\begin{corollary}
			\thlabel{corollary:omegazetaPart}
			$\{\omega,\zeta\}$ is not $\pl$-learnable. 
		\end{corollary}

		We now proceed to show that the reduction from $E_3$ to $\pl$ is strict. To do so, we consider the family of structures $\mathfrak{F}^*\defas \{\tilde{\omega},\tilde{\omega^*}\} \cup \{\tilde{L}_n: n \geq 2\}$, where $L_n$ is a finite linear order of $n$ elements.
		\thref{lemma:f*pl} and \thref{lemma:F_E_3} show respectively that $\mathfrak{F^*}$ is $\pl$-learnable and not $E_3$-learnable: together with \thref{proposition:E3pl} we obtain that $E_3\strictlearnreducible{} \pl$.
		
		\begin{lemma}
			\thlabel{lemma:f*pl}
			$\mathfrak{F}^*$ is $\pl$-learnable.
		\end{lemma}
		\begin{proof}
			Given $\esse \in \mathrm{LD}(\mathfrak{F}^*)$, define $min_s\defas \min_{\leq_{\esse}}\{n \in \esse\restriction_s: (\exists k \in \esse\restriction_s)(n <_{\mathcal{S}} k)\}$ 
			and $max_s\defas \max_{\leq_{\esse}}\{n \in \esse\restriction_s: (\exists k \in \esse\restriction_s)(k <_{\mathcal{S}} n)\}$ and we let $c_{min}(s)\defas |\{t \leq s : min_s=min_{t}\}|$ and $c_{max}(s)\defas |\{t \leq s : max_s=max_{t}\}|$. Then we define a learner $\learnerM$ such that $\learnerM(\esse\restriction_0) \defas \ulcorner \tilde{\omega} \urcorner$ and $\learnerM(\esse\restriction_{s+1})\defas$
			\[
			\begin{cases}
				\code{\tilde{L}_n} & \text{if } 
				
				n= \max\{m: L_m \hookrightarrow \esse\restriction_{s}\}=\max\{m: L_m \hookrightarrow \esse\restriction_{s+1}\}, \\
				\code{\tilde{\omega}} & \text{if } \max\{m: L_m \hookrightarrow \esse\restriction_{s}\} \neq \max\{m: L_m \hookrightarrow \esse\restriction_{s+1}\}
				
				\land c_{min}(s)\geq c_{max}(s)\\
				\code{\tilde{\omega^*}} & \text{if } \max\{m: L_m \hookrightarrow \esse\restriction_{s}\} \neq \max\{m: L_m \hookrightarrow \esse\restriction_{s+1}\}
				\land c_{min}(s)< c_{max}(s)
			\end{cases}\]
			We claim that $\learnerM$ $\pl$-learns $\mathfrak{F}^*$. Indeed, if $\esse\cong \tilde{L}_n$ for some $n \in \nats$, we obtain that $(\exists s_0)(\forall t\geq s_0)( L_n \hookrightarrow \esse\restriction_t \land L_{n+1} \not\hookrightarrow \esse\restriction_t)$, and hence, for all $t>s_0$, $\learnerM(\esse\restriction_t)=\code{\tilde{L}_n}$, i.e., beyond stage $s_0$, the first condition of $\learnerM$'s definition always applies. Otherwise, suppose that $\esse\cong \tilde{\omega}$ (the case for $\esse \cong \tilde{\omega^*}$ is analogous). We claim that
			\begin{enumerate}[(i)]
				\item there are infinitely many $t \in \nats$ such that $\learnerM(\esse\restriction_t)=\code{\tilde{\omega}}$ and
				\item for every $i \neq \code{\tilde{\omega}}$, $|\{s: \learnerM(\esse\restriction_s)=i\}|$ is finite.
			\end{enumerate}
			Both conditions together prove that $\learnerM$ $\pl$-learns $\mathfrak{F}^*$. Notice that $\tilde{\omega}$ has a least element but not a greatest one, hence $(\exists s_0)(\forall t\geq s_0)(c_{min}(t)> c_{\max}(t))$. Item \textbf{(i)} follows by combining the previous observation and the fact that there are infinitely many $s \in \nats$ such that \[\max\{m: L_m \hookrightarrow \esse\restriction_{s}\} \neq \max\{m: L_m \hookrightarrow \esse\restriction_{s+1}\},\] i.e., the second condition of $\learnerM$'s definition holds infinitely often. To prove the item \textbf{(ii)} we have two cases:
			\begin{itemize}
			\item if $i=\code{\tilde{\omega^*}}$, then $(\forall t\geq s_0)(c_{min}(t)> c_{max}(t))$, and hence, the third condition of $\learnerM$'s definition never applies from stage $s_0$ on.
			\item Otherwise,  if  $i=\code{\tilde{L}_n}$, then $| \{s:n=\max\{m: L_m \hookrightarrow \esse\restriction_{s}\}\} | = k_n$ for some $k_n \in \nats$, and hence, the first condition of $\learnerM$'s definition applies for at most $k_n$ times.
		\end{itemize}
			This concludes the proof.
		\end{proof}
		
		\begin{lemma}
			\thlabel{lemma:F_E_3}
			$\mathfrak{F}^*$ is not $E_3$-learnable.
		\end{lemma}
		\begin{proof}
			Towards a contradiction, assume that $\mathfrak{F}^*$ is $E_3$-learnable, and let the continuous reduction from $\ldofk$ to $E_3$ be witnessed by $\Gamma$. Without loss of generality, we may assume that $\Gamma(\tilde{\omega})^{[0]}\ \centernot{E_0}\ \Gamma(\tilde{\omega^*})^{[0]}$ (otherwise, just take the first column in which $\Gamma(\tilde{\omega})$ and $\Gamma(\tilde{\omega^*})$ disagree: such a column, by the definition of reduction, must exist). Notice also that for every $n>1$ and for every $\esse \in \ldofk$, the following holds: if $\esse \cong \tilde{L}_n$, then   $\Gamma(\tilde{L}_n)^{[0]}\ E_0\ \Gamma(\esse)^{[0]}$. 
			
			Since $\Gamma(\tilde{\omega})^{[0]}\ \centernot{E_0}\ \Gamma(\tilde{\omega^*})^{[0]}$, for every $\tilde{L}_n$, $\Gamma(\tilde{L}_n)^{[0]}\ \centernot{E_0}\ \Gamma(\tilde{\omega})^{[0]}$ or $\Gamma(\tilde{L}_n)^{[0]}\ \centernot{E_0}\ \Gamma(\tilde{\omega^*})^{[0]}$. Hence at least one set between $\{\tilde{L}_n : \Gamma(\tilde{L}_n)^{[0]}\ \centernot{E_0}\ \Gamma(\tilde{\omega})^{[0]}\}$ and $\{\tilde{L}_n : \Gamma(\tilde{L}_n)^{[0]}\ \centernot{E_0}\ \Gamma(\tilde{\omega^*})^{[0]}\}$ is infinite: without loss of generality assume the first one to be infinite, and let the elements in it be indexed as $\tilde{L}_{k_0}, \tilde{L}_{k_1}, \dots$. 
			
			We inductively define a copy $\esse$ of $\tilde{\omega}$ as follows. At stage $0$ add inside $\esse$ an isomorphic copy of $L_{k_0}$, and define $n_0\defas 0$. At the beginning of stage $s+1$, we assume that $\esse$ consists of an isomorphic copy of $L_{k_s}$ and possibly some elements incomparable with it. Then we extend the copy of $L_{k_s}$ with $(k_{s+1}-k_{s})$ fresh elements so that $\esse$ consists of an isomorphic copy of $L_{k_{s+1}}$ (where these fresh elements are added as greater than the existing elements of $L_{k_s}$) and some incomparable to the elements of $L_{k_{s+1}}$. Keep adding incomparable elements to $\esse$ and notice that, if we continue in this fashion, then $\esse\cong \tilde{L}_{k_{s+1}}$. Hence, since $\Gamma(\tilde{L}_{k_{s+1}})^{[0]}\ \centernot{E_0}\ \Gamma(\tilde{\omega})^{[0]}$, there must be an index $n_{s+1} > n_{s}$ such that  $\Gamma(\esse)^{[0]}(n_{s+1}) \neq  \Gamma(\tilde{\omega})^{[0]}(n_{s+1})$. When we find such an index, we proceed to stage $s+2$.
			
			To conclude the proof, notice that in the limit, we have that $\esse \cong \tilde{\omega}$. On the other hand, the construction of $\esse$ ensures that $\Gamma(\esse)^{[0]} \centernot{E_0} \Gamma(\tilde{\omega})^{[0]}$ as, at each stage $s$, we find a new input on which  $\Gamma(\esse)^{[0]}$ and  $\Gamma(\tilde{\omega})^{[0]}$ disagree. Therefore, the map $\Gamma$ cannot witness the $E_3$-learnability of $\mathfrak{F}^*$. This concludes the proof.
		\end{proof}

		\begin{corollary}
			$E_3  \strictlearnreducible{} \pl$.
		\end{corollary}
		We mention that \cite[Theorem 5.4]{bazhenov2021learning} provides a syntactic characterization of $E_3$-learnability.

		We conclude this section by stating explicitly the relation between $\pl$- and $\eset$-learning for finite families. 
		
		\begin{theorem}{\cite{ciprianimarconesanmauro}}
			\thlabel{theorem:esetcharacterization}
			A family of structures $\K$ is $\eset$-learnable if and only if $\K$ is a $\sigmainf{2}$-partial order.
		\end{theorem}
		Combining \thref{theorem:plcharacterization} and \thref{theorem:esetcharacterization} we obtain the following corollary.
		\begin{proposition}
			$\pl\strictlearnreducible{\finitary} \eset$.
		\end{proposition}

		\section{Conclusions}
		\label{sec:conclusions}
		This paper made additional contributions to the field of algorithmic learning for algebraic structures. Explanatory learning has so far received the most attention within the framework presented in \thref{definition:paradigm}. Here, we proposed the investigation of additional well-known learning paradigms, placing such paradigms in the hierarchy offered by the recent notion of $E$-learnability: a summary of the results can be found in the figure below.
		
		\cite[Theorem 3.1]{bazhenov2020learning} provided a syntactic characterization of $\ex$-learning using what is now referred to as a $\sigmainf{2}$-strong antichain, which proved to be crucial in determining when a family of structures is $\ex$-learnable. We successfully provided a syntactic characterization for all the new learning paradigms we introduced, and these characterizations turned out to be natural, as they can be described in terms of the inclusion of $\sigmainf{n}$-theories for $n \in \{1,2\}$. The existence of such natural characterizations demonstrates that both the learning paradigms derived from classical algorithmic learning theory and the $E$-learnabilities we explored are indeed natural.
		
		We suggest as further direction the one of exploring (and trying to characterize syntactically) other classical learning paradigms that were not considered in this paper. 
		We also mention that the study of this learning hierarchy is just at the beginning and we plan to address these and further directions in future studies. 
		
		\begin{figure}[H]
			
			\begin{tikzpicture}[scale=0.6]
				
				\node[draw, fill=gray!20] (a) at (0,0) {$E_{set}$  \; \begin{small} $(\Sigma^{\mathrm{inf}}_2 \, \text{-p.)}$ \end{small}};
				\node[draw, fill=gray!20] (b) at (0,-1.5) {$\mathbf{PL} $ \; \begin{small} $(\Sigma^{\mathrm{inf}}_2 \, \text{-a.)}$ \end{small}};
				\node[draw, fill=gray!20] (c) at (0,-3) {$E_3$ \; \begin{small}
		(\cite[Theorem 5.4]{bazhenov2021learning})\end{small}		
			
		};
				\node[draw, fill=gray!20] (d0) at (-4,-4.5) {$\erange$ \; \begin{small} $(\Sigma^{\mathrm{inf}}_1 \, \text{-p.)}$ \end{small}};
				\node[draw, fill=gray!20] (d1) at (4,-4.5) {$\mathbf{Ex},\, E_0, \, E_1, \, E_2$ \; \begin{small}
						$(\Sigma^{\mathrm{inf}}_2 \, \text{-s.a.}$)\end{small}};
				\node[draw, fill=gray!20] (e) at (0,-6) {$\mathbf{nUs}$, $\mathbf{Dec}$ \; \begin{small}$(solid  \ \sigmainf{1}\text{-p.})$\end{small}};
				
				\node[draw, fill=gray!20] (f) at (0,-7.5) {$\col$-learning \; \begin{small}
						$(\Sigma^{\mathrm{inf}}_1 \; \text{-a.}$)\end{small}};
				
				\node[draw, fill=gray!20] (g) at (0,-9) {$\Id$ \; \begin{small}			
	(\thref{theorem:idcharacterization})\end{small}		};
				
				\node[draw, fill=gray!20] (h) at (0,-10.5) {$\mathbf{Fin}, \, =_{\mathbb{N}}$ \; \begin{small}
						$(\Sigma^{\mathrm{inf}}_1 \, \text{-s.a.}$)\end{small}};
				
				\node[draw, fill=gray!20] (h1) at (14, -10.5) {$\mathbf{Fin}, \, =_{\mathbb{N}}, \, \Id, \, \col$-learning };
				
				\node[draw, fill=gray!20] (g1) at (14, -8.5) {$\nonushape, \, \decisive, \, \erange$ };
				
				\node[draw, fill=gray!20] (d3) at (14,-6.5) {$\mathbf{Ex},\, E_0, \, E_1, \, E_2, \, E_3, \, \pl$  };
				
				\node[draw, fill=gray!20] (a1) at (14,-4.5) {$E_{set}$  };
				
				\draw[->] (c)-- (b) -- (a);
				
				\draw[->] (d0) -- (c);
				\draw[->] (d1) -- (c);
				\draw[->] (e) -- (d0);
				\draw[->] (e) -- (d1);
				\draw[->] (h) -- (g) -- (f) -- (e);
				\draw[->] (h) -- (g);
				\draw[->] (g) -- (f);
				\draw[->] (h1) -- (g1);
				\draw[->] (g1) -- (d3);
				\draw[->] (d3) -- (a1);
			\end{tikzpicture}
			\caption{The learning paradigm considered in this paper. On the left-hand-side the picture refers to learn reducibility while the right-hand-side refers to finitary learn reducibility: The arrows represent (finite) learn reducibility in the direction of the arrow. For $n \in \{1,2\}$ $\sigmainf{n}$-a., -s.a., and -p. denote respectively $\sigmainf{n}$-antichains, -strong antichains and partial orders as defined in \thref{definition:sigmainfnposet}. }
		\end{figure}
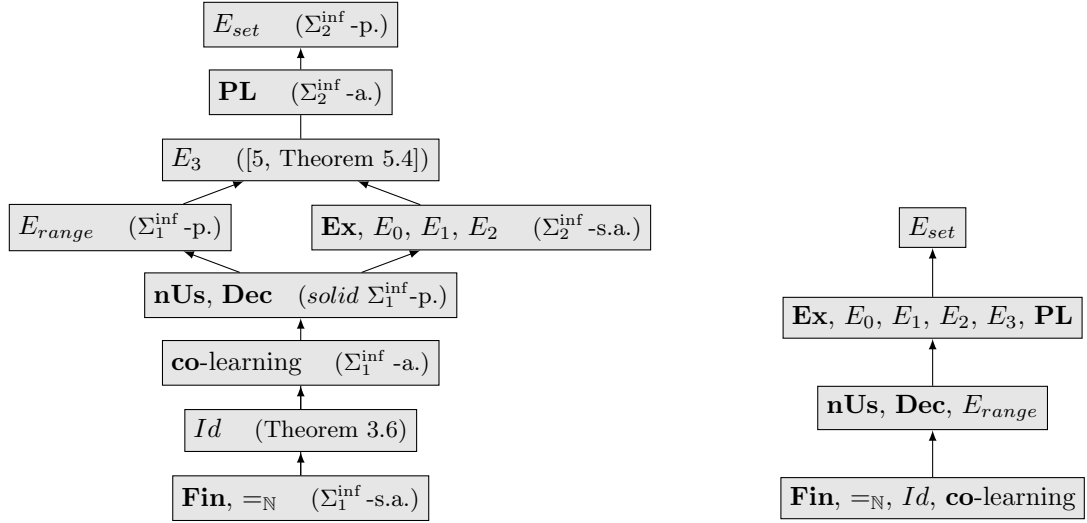

		\bibliographystyle{mbibstyle}
		\bibliography{references}
		
	\end{document}